\documentclass{amsart}

\usepackage{amsfonts, amsmath, amsthm,amssymb}
\usepackage[dvips]{graphicx}
\usepackage{amscd}

\makeatletter
\@addtoreset{equation}{section}

\makeatother

\author{Hideki Miyachi}

\address{Department of Mathematics,
Graduate School of Science, Osaka University,
Machikaneyama 1-1, Toyonaka, Osaka, 560-0043, Japan}
\email{miyachi@math.sci.osaka-u.ac.jp}

\newtheorem{theorem}{Theorem}
\newtheorem{proposition}{Proposition}[section]
\newtheorem{lemma}{Lemma}[section]
\newtheorem{corollary}{Corollary}
\newtheorem{claim}{Claim}[section]
\newtheorem{remark}{Remark}[section]
\newtheorem{question}{Question}

\makeatletter
\@addtoreset{equation}{section}

\makeatother

\usepackage{amsfonts, amsmath, amssymb, amsthm}

\newtheorem{convention}{Convention}

\newcommand{\ext}{{\rm Ext}}

\newcommand{\cl}[1]{{\rm cl}_{GM}(#1)}
\newcommand{\partialGM}[1]{{\partial_{GM}#1}}

\newcommand{\proj}{{\rm proj}}
\newcommand{\Teich}[2]{\mathcal{T}_{#1,#2}}

\newcommand{\TeichbX}{\mathcal{T}_{g,m}}

\newcommand{\GmInv}{\mathcal{C}_{GM}}
\newcommand{\GmTeich}{\mathcal{T}_{GM}}
\newcommand{\Gmbdy}{\tilde{\partial}_{GM}}

\newcommand{\mf}{\mathcal{MF}}
\newcommand{\pmf}{\mathcal{PMF}}

\newcommand{\extB}{{\mathcal Ext}}
\newcommand{\GmInvC}{{\sf M}\mathcal{C}_{GM}}
\newcommand{\GmTeichC}{{\sf M}\mathcal{T}_{GM}}
\newcommand{\GmbdyC}{{\sf M}\tilde{\partial}_{GM}}
\newcommand{\mfC}{{\sf MF}}
\newcommand{\pmfC}{{\sf MF}_1}
\newcommand{\gromov}[3]{\langle #1\,|\,#2\rangle_{#3}}

\newcommand{\mcg}{{\rm Mod}}


\begin{document}

\title[Unification of extremal length geometry]{Unification of extremal length geometry on Teichm\"uller space via
intersection number}

\thanks{The author is partially supported by the Ministry of Education, Science, Sports and Culture, Grant-in-Aid for Scientific Research (C), 21540177.}






\maketitle

\begin{abstract}
In this paper,
we give a framework for the study of
the extremal length geometry of Teichm\"uller space
after S. Kerckhoff, F. Gardiner and H. Masur.
There is a natural compactification using extremal length geometry introduced
by Gardiner and Masur.
The compactification is realized in a certain projective space.
We develop the extremal length geometry in the cone
which is defined as the inverse image of
the compactification via the quotient mapping.
The compactification is identified with a subset of the cone
by taking an appropriate lift.
The cone contains canonically the space of measured foliations
in the boundary.

We first extend the geometric intersection number
on the space of measured foliations to the cone,
and observe that the restriction of the intersection
number to Teichm\"uller space
is represented explicitly by the formula in terms of
the Gromov product with respect to the Teichm\"uller distance.
From this observation,
we deduce that the Gromov product extends continuously
to the compactification.

As an application,
we obtain an alternative approach to
Earle-Ivanov-Kra-Markovic-Royden's
characterization of isometries.
Namely,
with some few exceptions,
the isometry group of Teichm\"uller space
with respect to the Teichm\"uller distance
is canonically isomorphic to the extended mapping class group.
We also obtain a new realization of Teichm\"uller space,
a hyperboloid model of Teichm\"uller space
with respect to the Teichm\"uller distance.
\end{abstract}

\tableofcontents

\section{Introduction}

The \emph{Teichm\"uller distance} is a canonical
and important distance on Teichm\"uller space.
The geometry of the Teichm\"uller distance
is deeply related to
the extremal length geometry on that space
(cf. \cite{MasurSurvey}).
To the author's knowledge,
 in \cite{Ker},
 S. Kerckhoff first studied
 the boundary of Teichm\"uller space at infinity via
extremal length.
The extremal length geometry on Teichm\"uller space
was formulated precisely by F. Gardiner and H. Masur in \cite{GM}.
(cf. \S\ref{subsec:results}).

%

\subsection{Motivation}
\label{subsec:motivation}
\paragraph{{\bf Unification of Teichm\"uller geometry in terms of intersection number}}
To define the Thurston compactification of Teichm\"uller space,
we first recognize each point of Teichm\"uller space as a function 
on the set of simple closed curves by assigning
the hyperbolic lengths of geodesic representatives,
and then,
we take the closure of the set of projective classes
of such functions in the projective space
(cf. \cite{FLP}).
In a broad sense,
completions due to Thurston carry out with recognizing each point of
Teichm\"uller space as a function on the set of simple closed curves
(see also \cite{DLR}).
Hence,
the Gardiner-Masur compactification is considered
as an object in the category ``Thurston's completion"
(cf. \eqref{eq:GM-embedding}).
Thus,
it is expected
that every boundary point of the Gardiner-Masur compactification
is recognized as the projective class of a function
defined by (a kind of) intersection number.

In \cite{Bonahon_currents}, F. Bonahon realized the Thurston compactification
in the space of geodesic currents.
Indeed, in his method,
any point of Teichm\"uller space is associated to an equivariant Radon measure
on the space of hyperbolic geodesics on the universal cover of the base surface
of Teichm\"uller space.
He extended the intersection number function
to the space of geodesic currents,
and gave a unified treatment for the Thurston compactification
in terms of the intersection number.
His theory is broadly applied in many fields in mathematics
and yields enormous rich results
(cf. e.g.  \cite{Bonahon_bouts} and \cite{DLR}).
Thus,
it is natural to ask:
\begin{question}
\label{question:1}
Can we develop
extremal length geometry in terms of intersection number?
\end{question}

\paragraph{{\bf Relation to the geometry of the Teichm\"uller distance}}
As discussed in the previous section,
the space $\mathbb{R}_+^{\mathcal{S}}$ of non-negative functions
on the set $\mathcal{S}$ of simple closed curves
is the ambient space of Thurston's completions.
The interior $(0,\infty)^{\mathcal{S}}$ of
$\mathbb{R}_+^{\mathcal{S}}$
admits a pseudo-distance
\begin{equation}
\label{eq:product_metrics_RS}
d_\infty(f,g)=\log \sup_{\alpha\in \mathcal{S}}
\left\{
\frac{f(\alpha)}{g(\alpha)},
\frac{g(\alpha)}{f(\alpha)}
\right\}
\end{equation}
which is perceived as
the product distance of countably many $1$-dimensional hyperbolic spaces.
Possibly $d_\infty(f,g)=\infty$ for some $f,g\in \mathbb{R}_+^{\mathcal{S}}$
and 
the topology from \eqref{eq:product_metrics_RS}
is different from
the product topology on $\mathbb{R}_+^{\mathcal{S}}$.
From Kerckhoff's formula \eqref{eq:Kerckhoff_formula},
a natural lift given in
\eqref{eq:GM_embedding_lift}
of the Gardiner-Masur embedding gives an isometric embedding
from Teichm\"uller space to the ambient space
$((0,\infty)^{\mathcal{S}},d_\infty)$.
One may ask:

\begin{question}
\label{question:2}
How is the geometry of Teichm\"uller distance
related to the geometry of the Gardiner-Masur compactification
(embedding)?
\end{question}


\subsection{Results}
\label{subsec:results}
In this paper,
we attempt to unify the extremal length geometry
via intersection number,
aiming for a counterpart for Bonahon's theory
on geodesic currents.



We fix the notation to give our results precisely.
Henceforth,
we fix a Riemann surface $X=X_{g,m}$ 
of genus $g$ with $m$ punctures such that $2g-2+m>0$.
Denote by $\TeichbX$ the Teichm\"uller space of $X$.
When the argument depends on the basepoint,
we consider the Teichm\"uller space
$\TeichbX$ as a pointed space
$(\Teich{g}{m},x_0)$,
where $x_0=(X,id)$.

Let $\mathcal{S}$ be the set of non-peripheral and
non-trivial simple closed curves on $X$,
and $\mf$ the space of measured foliations.
The space $\mf$ is contained in
$\mathbb{R}_+^{\mathcal{S}}$
(cf. \S\ref{subsec:measured_foliation}).

We refer readers to \S\ref{sec:GM_closure}
for details on the Gardiner-Masur closure.
We consider the cone $\GmInv$
which is defined as the inverse image
of the Gardiner-Masur closure $\cl{\Teich{g}{m}}$ via the projection
$\mathbb{R}_+^{\mathcal{S}}-\{0\}\to {\rm P}\mathbb{R}_+^{\mathcal{S}}$
(cf. \S\ref{subsec:cones}).
It is known that  the space
$\pmf$ of projective measured foliations
is contained in
the Gardiner-Masur boundary
$\partialGM{\Teich{g}{m}}$
and hence $\mf\subset \GmInv$ (cf. \cite{GM}).
One of our aims in this paper is to define
the intersection number function on $\GmInv$.
In order to avoid any confusion,
we denote by $I(\,\cdot\,,\,\cdot\,)$ the original geometric intersection number function on $\mf$.

\subsubsection{Unification by intersection number}
The Gardiner-Masur embedding \eqref{eq:GM-embedding}
admits a natural lift
\begin{equation}
\label{eq:GM_embedding_lift}
\tilde{\Phi}_{GM}\colon
\Teich{g}{m}\ni y\mapsto
[\mathcal{S}\ni \alpha\mapsto \ext_y(\alpha)^{1/2}]
\in \GmInv\subset \mathbb{R}_+^\mathcal{S}.
\end{equation}
Our unification is stated as follows.
%

\begin{theorem}[Unification]
\label{thm:main_realization_0}
There is a unique continuous function
$$
i(\,\cdot\,,\,\cdot\,)\colon \GmInv\times \GmInv\to \mathbb{R}
$$
with the following properties.
\begin{itemize}
\item[{\rm (i)}]
For any $y\in \Teich{g}{m}$,
the projective class of
the function $\mathcal{S}\ni\alpha\mapsto i(\tilde{\Phi}_{GM}(y),\alpha)$
is exactly the image of $y$ under the Gardiner-Masur embedding.
Actually,
it holds
$$
 i(\tilde{\Phi}_{GM}(y),\alpha)=\ext_y(\alpha)^{1/2}
$$
for all $\alpha\in \mathcal{S}$.
\item[{\rm (ii)}]
For $\mathfrak{a},\mathfrak{b}\in \GmInv$,
$i(\mathfrak{a},\mathfrak{b})=i(\mathfrak{b},\mathfrak{a})$.
\item[{\rm (iii)}]
For $\mathfrak{a},\mathfrak{b}\in \GmInv$ and $t,s\ge 0$,
$i(t\mathfrak{a},s\mathfrak{b})=ts\,i(\mathfrak{a},\mathfrak{b})$.
\item[{\rm (iv)}]
For any $y,z\in \Teich{g}{m}$,
$$
i(\tilde{\Phi}_{GM}(y),\tilde{\Phi}_{GM}(z))=\exp(d_T(y,z)).
$$
In particular,
we have $i(\tilde{\Phi}_{GM}(y),\tilde{\Phi}_{GM}(y))=1$
for $y\in \Teich{g}{m}$.
\item[{\rm (v)}]
For $F,G\in \mf\subset \GmInv$,
the value $i(F,G)$ is equal to the geometric intersection number
$I(F,G)$.
\end{itemize}
\end{theorem}

For a technical reason,
instead of proving Theorem \ref{thm:main_realization_0},
we will show Theorem \ref{thm:pairing_intrinsic},
which is the \emph{basepoint dependent} version
of Theorem \ref{thm:main_realization_0}
(cf.  \S\ref{sec:extension_intersection_number}).
Actually,
we will consider another lift
\begin{equation}
\label{eq:Lift_GM}
\Psi_{x_0}\colon
\Teich{g}{m}\ni y\mapsto 
\left[
\mathcal{S}\ni\alpha
\mapsto 
\exp(-d_T(x_0,y))\cdot
\ext_y(\alpha)^{1/2}
\right]
\in \GmInv
\end{equation}
of the Gardiner-Masur embedding
in Theorem \ref{thm:pairing_intrinsic}
in place of $\tilde{\Phi}_{GM}$.
Namely,
\begin{equation}
\label{eq:Psi_and_Phi}
\Psi_{x_0}(y)=\exp(-d_T(x_0,y))\cdot \tilde{\Phi}_{GM}(y)
\end{equation}
for all $y\in \Teich{g}{m}$.
One of advantages to use the embedding $\Psi_{x_0}$
is that $\Psi_{x_0}$
admits a continuous extension to $\cl{\TeichbX}$,
whereas $\tilde{\Phi}_{GM}$ diverges at infinity
(cf. Proposition \ref{prop:continuity_E_p} and \eqref{eq:Lift_GM}).


\subsubsection{Hyperboloid model of Teichm\"uller space}
\label{subsubsec:hyperboliod_model}
We represent the situations of our theorems schematically in Figure \ref{fig:cone}.
For any $y\in \TeichbX$,
$\tilde{\Phi}_{GM}(y)$ and $\Psi_{x_0}(y)$ are projectively equivalent
in $\mathbb{R}_+^{\mathcal{S}}$.
From (iv) in Theorem \ref{thm:main_realization_0},
the image under $\tilde{\Phi}_{GM}$ coincides with
the ``hyperboloid"
\begin{equation}
\label{eq:hyperboloid_model_T}
\{\mathfrak{a}\in \GmInv\mid i(\mathfrak{a},\mathfrak{a})=1\},
\end{equation}
and the boundary of the cone $\GmInv$ is represented as
the ``light cone"
\begin{equation}
\label{eq:light_cone_T}
\{\mathfrak{a}\in \GmInv\mid i(\mathfrak{a},\mathfrak{a})=0\}
\end{equation}
from (iii) and (iv) in Theorem \ref{thm:main_realization_0}
and the continuity of the intersection number on $\GmInv$.
The image of $\Psi_{x_0}$ looks like a section in the cone.
\begin{figure}
\begin{center}
\includegraphics[height=5cm]{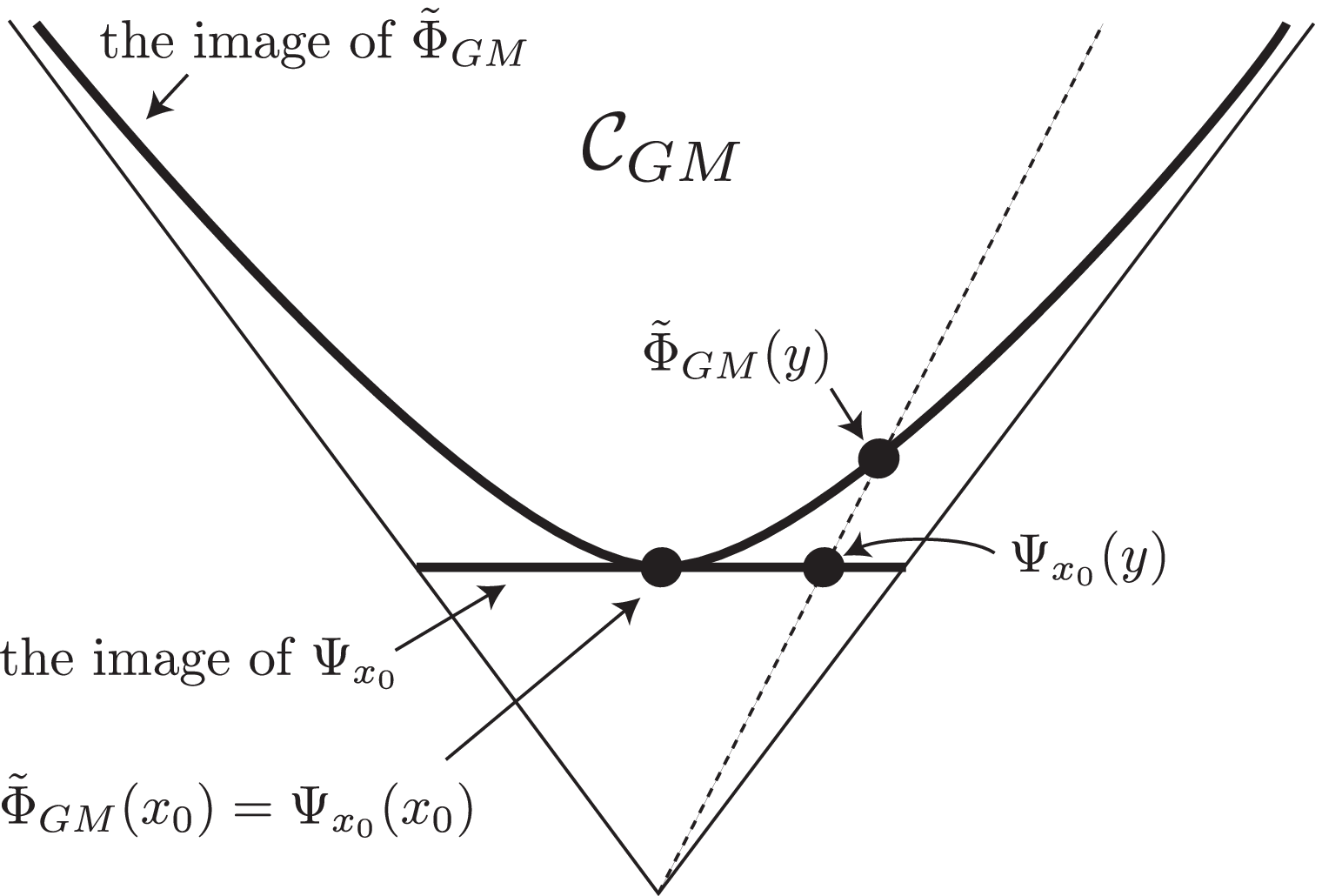}
\caption{Cone $\GmInv$ and the images of $\tilde{\Phi}_{GM}$ and $\Psi_{x_0}$.}
\label{fig:cone}
\end{center}
\end{figure}
These images contact only at the images of the basepoint.
In the hyperboloid model,
the Teichm\"uller distance
is represented as
$$
d_T(y,z)=\log\,i(\tilde{\Phi}_{GM}(y),\tilde{\Phi}_{GM}(z)).
$$
%

This hyperboloid model might be a comparable object with Bonahon's realization
of the Thurston compactification
of Teichm\"uller space in the space of geodesic currents
(cf. \S4 in \cite{Bonahon_currents}).

\subsubsection{Extension of the Gromov product}
The following corollary confirms that the Gardiner-Masur boundary
is a kind of a canonical boundary for the geometry of the Teichm\"uller distance.

\begin{corollary}[Extension of the Gromov product for $d_T$]
\label{coro:Gromov_product_boundary}
For any $x_0\in \Teich{g}{m}$,
there is a unique continuous function
$$
\gromov{\cdot}{\cdot}{x_0}
\colon
\cl{\Teich{g}{m}}\times \cl{\Teich{g}{m}}
\to [0,+\infty]
$$
such that
\begin{itemize}
\item[{\rm (1)}]
for $y,z\in \Teich{g}{m}$,
$$
\gromov{y}{z}{x_0}
=\frac{1}{2}
(d_T(x_0,y)+d_T(x_0,z)-d_T(y,z)),
$$
\item[{\rm (2)}]
for $[F],[G]\in \pmf\subset \partialGM{\Teich{g}{m}}$,
$$
\exp(-2\gromov{[F]}{[G]}{x_0})=
\frac{I(F,G)}{\ext_{x_0}(F)^{1/2}\cdot \ext_{x_0}(G)^{1/2}}. 
$$
\end{itemize}
\end{corollary}

The conclusion in
Corollary \ref{coro:Gromov_product_boundary} is somewhat surprising
because Teichm\"uller space with the Teichm\"uller distance
is believed to be a metric space with less ``good natures"
for geodesic triangles.
For instance,
Teichm\"uller space is neither a metric space with
Busemann negative curvature
nor a Gromov hyperbolic space
(cf. \cite{Masur0}, \cite{MW2} and \cite{MP}).
Recently,
C. Walsh informed that there is a geodesic metric space
with the property that the Gromov product 
does not extend to the horofunction boundary
(cf. \S\ref{sec:non-extendability-GromovProduct}).

\subsubsection{Rigidity theorem for mappings of bounded distortion for triangles}
\label{subsubsec:almost_isometry}
Our unified treatment of extremal length geometry
in terms of intersection number
links the geometry of the Teichm\"uller distance
(an analytical aspect in Teichm\"uller theory)
with the geometry on $\mf$ via intersection number
(a topological aspect in Teichm\"uller theory).
%

We will deal with a \emph{mapping of bounded distortion for triangles}
which is defined as a mapping $\omega:\Teich{g}{m}\to \Teich{g}{m}$
satisfying
$$
\frac{1}{D_1}\gromov{x}{y}{z}-D_2
\le
\gromov{\omega(x)}{\omega(y)}{\omega(z)}
\le
D_1
\gromov{x}{y}{z}+D_2
$$
for all $x,y,z\in \Teich{g}{m}$ and some constants $D_1,D_2>0$ independent
of the choice of points of $\Teich{g}{m}$.
A mapping $\omega'\colon \Teich{g}{m}\to \Teich{g}{m}$ is said to be
a \emph{quasi-inverse} of a mapping $\omega\colon \Teich{g}{m}\to \Teich{g}{m}$ if 
there is a constant $D_3>0$ such that
$$
\sup_{x\in \Teich{g}{m}}
\{d_T(x,\omega\circ \omega'(x)),d_T(x,\omega'\circ \omega(x))\}\le D_3.
$$
One can easily check that
any quasi-inverse $\omega'$ of $\omega$ is also a mapping
of bounded distortion for triangles.
In \S\ref{subsec:proof_of_theorem_characterization},
we prove the following.

\begin{theorem}[Asymptotic Rigidity]
\label{thm:almost_isometry}
Suppose that the complex dimension of $\Teich{g}{m}$ is at least two.
Let $\omega\colon \Teich{g}{m}\to \Teich{g}{m}$ be a mapping
of bounded distortion for triangles.
Assume the following two conditions:
\begin{itemize}
\item[{\rm (a)}]
The map $\omega$ admits a continuous extension to $\partialGM{\Teich{g}{m}}$.
\item[{\rm (b)}]
The map $\omega$ has a quasi-inverse $\omega'$ which admits a continuous extension
to $\partialGM{\Teich{g}{m}}$.
\end{itemize}
Then,
the following hold:
\begin{itemize}
\item[{\rm (1)}]
The map $\omega$ acts homeomorphically on
$\pmf\subset \partialGM{\Teich{g}{m}}$
and $\omega'=\omega^{-1}$ on $\pmf$.
\item[{\rm (2)}]
The restriction of $\omega$ to $\pmf$
preserves $\mathcal{S}$ and induces a simplicial automorphism of
the complex of curves.
\end{itemize}
\end{theorem}
%
By definition,
a quasi-invertible mapping of bounded distortion for triangles is a quasi-isometry.
However,
the author does not know whether Theorem \ref{thm:almost_isometry}
holds for quasi-isometries on $\Teich{g}{m}$.
We remark that
(1) in Theorem \ref{thm:almost_isometry} holds
when the complex dimension of $\Teich{g}{m}$ is equal to one.
In this case,
$(\Teich{g}{m},d_T)$ is isometric to the hyperbolic plane,
and both the Gardiner-Masur boundary and $\pmf$
coincide with the boundary at infinity
of the hyperbolic plane (cf. e.g. \cite{Mi0}).
Hence any quasi-isometry on $(\Teich{g}{m},d_T)$ induces
a homeomorphism of $\pmf$.
However,
the assertion (2) does not hold because the isometry group
of $(\Teich{g}{m},d_T)$ acts transitively in this case.

\subsubsection{Isometries on $\Teich{g}{m}$}
\label{subsubsec:isometry-teich}
%
Theorem \ref{thm:almost_isometry}
allows us to give
an alternative approach to
Earle-Ivanov-Kra-Markovic-Royden's characterization of
the isometry group of $(\Teich{g}{m},d_T)$
via the Gardiner-Masur compactification.
Namely,
we show the following
in \S\ref{subsec:isometry_on_TX}.
 

\begin{corollary}
[Royden \cite{Royden},
Earle-Kra \cite{EK},
Ivanov \cite{Ivanov},
and
Earle-Markovic \cite{EM}]
\label{coro:Ivanov_characterization}
Suppose that
$3g-3+m\ge 2$
and $(g,m)$ is neither $(1,2)$ nor $(2,0)$.
Then,
the isometry group of $(\Teich{g}{m},d_T)$
is canonically isomorphic to the extended mapping class group.
\end{corollary}

Actually,
our proof of Corollary \ref{coro:Ivanov_characterization}
is somewhat modelled on Ivanov's proof.
We outline the idea of his proof.
The essential part is to show that an isometric action on $(\Teich{g}{m},d_T)$
induces an automorphism of the complex of curves.
After then,
from a theorem by Ivanov, Korkmaz and Luo,
we see that such an automorphism of the complex of curves is induced by an
element of the extended mapping class group
(cf. \cite{Ivanov0}, \cite{Korkmaz} and \cite{Luo}).
Finally, it is checked that the action of the given isometry coincides with the action of
the element of the extended mapping class group.

As noted before,
our proof of Corollary \ref{coro:Ivanov_characterization} also 
follows the same line.
However,
our proof of the essential part above follows from Theorem \ref{thm:almost_isometry}
which holds for
mappings of bounded distortion for triangles.
To show the essential part above,
Ivanov induces a self-homeomorphism of $\pmf$.
To do this,
he identifies $\pmf$ with the unit sphere in the tangent space,
and defines the self-homeomorphism
by passing the ``exponential maps"
(cf. the discussion after the proof of Lemma 5.2 in \cite{Ivanov}).

\subsection{Plan of this paper}
\label{subsec:plan}
%
This paper is organized as as follows.
In \S\S\ref{sec:Teichmuller-theory} and \ref{sec:GM_closure},
we recall basic notions in Teichm\"uller theory
and known results for the Gardiner-Masur compactification.

In \S\ref{sec:cones},
we define the cones which are essential objects in this paper.
We also define the (topological) \emph{models} of cones,
and canonical identifications between cones and their models.
We use such models when we develop an argument which
depends on the choice of the basepoint of $\TeichbX$.

From \S\ref{sec:extremal_length_GmInv} to \S\ref{sec:extension_intersection_number},
we devote to constructing the intersection number on the cone $\GmInv$.
%
In \S\ref{sec:extremal_length_GmInv},
we define the extremal length $\extB^{x_0}_{\cdot}(\,\cdot\,)$ and
the intersection number $i_{x_0}(\,\cdot\,,\,\cdot\,)$
associated to the basepoint $x_0$ on a part of each model.
The definition of this ``new" extremal length
is motivated by the following formula
\begin{equation}
\label{eq:tautological_formula}
\ext_y(G)=\sup_{F\in \mf-\{0\}}
\frac{I(G,F)^2}{\ext_y(F)}
\end{equation}
for $G\in \mf$
(cf. \eqref{eq:minsky_inequality}).
We first define the intersection number between elements of $\GmInv$
and measured foliations (\S\ref{subsec:intersection_number_associated_to_basepoint}),
and
the extremal length for elements of $\GmInv$
(\S\ref{subsec:extremal_length_GmInv}).
%
%
In \S\ref{sec:topology_of_models},
we discuss the topology of models of cones.
In  \S\S\ref{sec:Gromov_product_extension}
and \ref{sec:extension_intersection_number},
the intersection number
on the cone is defined by extending the functions defined in earlier sections.
We prove Corollary \ref{coro:Gromov_product_boundary}
in \S\ref{subsec:extension_Gromov}.
In \S\ref{sec:isometric_action},
we show Theorem \ref{thm:almost_isometry}
and give an alternative approach to Earle-Kra-Ivanov-Markovic-Royden's characterization in Corollary \ref{coro:Ivanov_characterization}.
%
%

\section{Teichm\"uller theory}
\label{sec:Teichmuller-theory}
\subsection{Teichm\"uller space}
\label{subsec:teichmuller-space}
The \emph{Teichm\"uller space 
$\Teich{g}{m}$ of Riemann surfaces of analytically finite type $(g,m)$}
is the set of equivalence classes of marked Riemann surfaces
$(Y,f)$ where $Y$ is a Riemann surface and $f:X\to Y$ a quasiconformal mapping.
Two marked Riemann surfaces $(Y_1,f_1)$
and $(Y_2,f_2)$ are
said to be \emph{Teichm\"uller equivalent}
if there is a conformal mapping $h:Y_1\to Y_2$
which is homotopic to $f_2\circ f_1^{-1}$.

Teichm\"uller space 
$\Teich{g}{m}$
has a canonical complete distance,
called the
\emph{Teichm\"uller distance $d_T$},
which is defined by
\begin{equation}
\label{eq:original_teichmuller_distance}
d_T(y_1,y_2)=\frac{1}{2}
\log
\inf
\{K(h)\mid \mbox{$h$ is q.c. homotopic to $f_2\circ f_1^{-1}$}
\}
\end{equation}
for $y_i=(Y_i,f_i)\in \Teich{g}{m}$
($i=1,2$),
where
$K(h)$ is the maximal dilatation of $h$ (e.g. \cite[\S4.1.1]{IT}).

%

\subsection{Measured foliations}
\label{subsec:measured_foliation}
Denote by $\mathbb{R}_+\otimes \mathcal{S}$ 
the set of formal products $t\alpha$ where $t\ge 0$ and
$\alpha\in \mathcal{S}$.
The set $\mathbb{R}_+\otimes \mathcal{S}$ is embedded into
$\mathbb{R}_+^\mathcal{S}$ by
\begin{equation}
\label{eq:scc_embedded}
\mathbb{R}_+\otimes \mathcal{S}\ni t\alpha\mapsto
[\mathcal{S}\ni \beta \mapsto t\,I(\alpha,\beta)]
\in \mathbb{R}_+^\mathcal{S}.
\end{equation}
We topologize $\mathbb{R}_+^\mathcal{S}$
with the pointwise convergence
(i.e. the product topology).
The \emph{space $\mf$
of measured foliations on $X$}
is the closure of the image of the mapping \eqref{eq:scc_embedded}.
The intersection number of any two weighted curves in
$\mathbb{R}_+\otimes \mathcal{S}$
is defined by $I(t\alpha,s\beta)=ts\,I(\alpha,\beta)$.
It is known that the intersection number function extends continuously
on $\mf\times \mf$
(cf. \cite{Rees}).

The positive numbers $\mathbb{R}_{>0}$ acts on $\mathbb{R}_+^\mathcal{S}$
by multiplication.
Let
\begin{equation}
\label{eq:quotient}
\proj\colon \mathbb{R}_+^\mathcal{S}-\{0\}\to {\rm P}\mathbb{R}_+^\mathcal{S}
=(\mathbb{R}_+^\mathcal{S}-\{0\})/\mathbb{R}_{>0}
\end{equation}
be the quotient mapping.
The space $\pmf$ of \emph{projective measured foliations}
is defined to be the quotient
$$
\pmf=\proj(\mf-\{0\})=(\mf-\{0\})/\mathbb{R}_{>0}.
$$
It is known that $\mf$ and $\pmf$
are homeomorphic to $\mathbb{R}^{6g-6+2n}$ and $S^{6g-7+2n}$
respectively (cf. \cite{FLP}).

\subsection{Extremal length}
\label{subsec:extremallength}
For $y=(Y,f)\in \Teich{g}{m}$ and $\alpha\in \mathcal{S}$,
the \emph{extremal length of $\alpha$ on $y$} is defined by
\begin{equation}
\label{eq:extremal_length_original}
\ext_y(\alpha)=1/\sup_A\{{\rm Mod}(A)\mid
\mbox{$A\subset Y$ and the core is homotopic to $f(\alpha)$}
\},
\end{equation}
where ${\rm Mod}(A)$ is the \emph{modulus} of an annulus $A$,
which is equal to $(\log r)/2\pi$ if $A$ is conformally equivalent
to a round annulus $\{1<|z|<r\}$.
For $t\alpha\in \mathbb{R}_+\otimes \mathcal{S}$,
we set
$$
\ext_y(t\alpha)=t^2\ext_y(\alpha).
$$
In \cite{Ker},
Kerckhoff showed that the extremal length function extends continuously 
on $\mf$.
Let
\begin{equation}
\label{eq:MF1}
\mf_1=\{F\in \mf\mid \ext_{x_0}(F)=1\}.
\end{equation}
The extremal length of measured foliations satisfies the following
inequality,
which is called
\emph{Minsky's inequality}:
\begin{equation}
\label{eq:minsky_inequality}
I(F,G)^2\le \ext_y(F)\cdot \ext_y(G)
\end{equation}
for all $y\in \Teich{g}{m}$ and $F,G\in \mf$
(cf. \cite{Minsky2}).
Minsky's inequality is sharp in the sense that 
for any $y\in \Teich{g}{m}$ and $F\in \mf-\{0\}$,
there is a unique $G\in \mf-\{0\}$
up to positive multiple
such that
\begin{equation}
\label{eq:minsky_inequality_sharp}
I(F,G)^2=\ext_y(F)\cdot \ext_y(G).
\end{equation}
Furthermore,
such a pair $F$ and $G$ of measured foliations
are realized by the horizontal and vertical foliations of
a holomorphic quadratic
differential on a marked Riemann surface $y$,
and vice versa
(cf. \cite{GM}).

\subsection{Kerckhoff's formula}
\label{subsec:kerckhoff_formula}
In \cite{Ker},
Kerckhoff 
gave the following formula:
\begin{equation}
\label{eq:Kerckhoff_formula}
d_T(y,z)=
\frac{1}{2}\log
\sup_{F\in \mf-\{0\}}
\frac{\ext_y(F)}{\ext_z(F)}
=
\frac{1}{2}\log
\max_{F\in \mf_1}
\frac{\ext_y(F)}{\ext_z(F)}.
\end{equation}
In fact,
for any $y_1,y_2\in \Teich{g}{m}$,
there is a unique pair $(F,G)$ of measured foliations in $\mf_1$
such that
\begin{equation}
\label{eq:kerckhoff_formula_sharp}
\frac{\ext_{y_1}(F)}{\ext_{y_2}(F)}
=\frac{\ext_{y_2}(G)}{\ext_{y_1}(G)}
=e^{2d_T(y_1,y_2)}.
\end{equation}

\section{The Gardiner-Masur closure}
\label{sec:GM_closure}
%
%
In \cite{GM},
Gardiner and Masur proved
that a mapping
\begin{equation}
\label{eq:GM-embedding}
\Phi_{GM}\colon
\Teich{g}{m}\ni y\mapsto
[\mathcal{S}\ni \alpha\mapsto \ext_y(\alpha)^{1/2}]
\in {\rm P}\mathbb{R}_+^\mathcal{S}
\end{equation}
is an embedding and the image is relatively compact,
where $\ext_y(\alpha)$ is the extremal length of
$\alpha\in \mathcal{S}$ on $y\in \Teich{g}{m}$.
The closure
$\cl{\Teich{g}{m}}$
of the image
is called the \emph{Gardiner-Masur closure} or \emph{compactification},
and the complement of the image in the closure is said to be the
\emph{Gardiner-Masur boundary}
which we denote by 
$\partialGM{\Teich{g}{m}}$.



For $y\in \TeichbX$,
we define a continuous function $\mathcal{E}_y$ on $\mf$
by
\begin{equation}
\label{eq:E_y}
\mathcal{E}_y(F)=
\left\{
\frac{\ext_{y}(F)}{K_y}
\right\}^{1/2}
\end{equation}
where $K_y=\exp(2d_T(x_0,y))$.
In \cite{Mi1},
the author showed that for any $p\in \partialGM{\TeichbX}$,
there is a continuous function $\mathcal{E}_p$ on $\mf$
such that
\begin{itemize}
\item[(E1)]
the projective class  of the assignment
$\mathcal{S}\ni \alpha\mapsto \mathcal{E}_p(\alpha)$ is equal to $p$;
\item[(E2)]
if a sequence $\{y_n\}_{n=1}^\infty$ converges to $p\in \cl{\TeichbX}$,
there are $t_0>0$ and a subsequence
$\{y_{n_j}\}_j$ such that $\mathcal{E}_{y_{n_j}}$ converges to
$t_0\,\mathcal{E}_p$ uniformly on any compact set of $\mf$.
\end{itemize}
Notice
that
$t\,\mathcal{E}_p$ also satisfies (E1) and (E2) above
for all $t>0$ and $p\in \partialGM{\TeichbX}$,
and the function $\mathcal{E}_p$ depends on the choice of basepoint $x_0$.
When we emphasis the dependence,
we write $\mathcal{E}^{x_0}_p$ instead of $\mathcal{E}_p$.

We first sharpen the condition (E2) above as follows (cf. \cite{Mi4}).

\begin{proposition}
\label{prop:continuity_E_p}
For any $p\in \partialGM{\TeichbX}$,
one can choose $\mathcal{E}_p$ appropriately such that
the function
$$
\cl{\TeichbX}\times \mf\ni (p,F)\mapsto \mathcal{E}_p(F)
$$
is continuous.
\end{proposition}

\begin{proof}
%
We normalize $\mathcal{E}_p$ such that
\begin{equation} \label{eq:maximal_E_p}
\max_{F\in \mf_1}
\mathcal{E}_p(F)=1.
\end{equation}
Notice from \eqref{eq:kerckhoff_formula_sharp}
that $\max_{F\in \mf_1}\mathcal{E}_y(F)=1$
for all $y\in \TeichbX$.
Let $\{y_n\}_{n=1}^\infty$ be a sequence
that converges to $p\in \partialGM{\TeichbX}$.
From the condition (E2) above,
there are a subsequence $\{y_{n_j}\}_j$ and $t_0>0$
such that $\mathcal{E}_{y_{n_j}}$ converges to $t_0\mathcal{E}_p$
uniformly on any compact set of $\mf$,
and hence
$$
1=\max_{F\in \mf_1}\mathcal{E}_{y_{n_j}}(F)\to
t_0\,\max_{F\in \mf_1}\mathcal{E}_p(F)=t_0.
$$
This implies that $\mathcal{E}_{y_n}$ converges to
$\mathcal{E}_p$ on any compact set of $\mf$.
\qed
\end{proof}

\begin{convention}
\label{convention:2}
In what follows,
we normalize $\mathcal{E}_p$ as in \eqref{eq:maximal_E_p}
for all $p\in \partialGM{\TeichbX}$.
\end{convention}


For instance,
for $G\in \mf$ it holds
\begin{equation}
\label{eq:E-p-For-measured-foliation}
\mathcal{E}_{[G]}(F)=\mathcal{E}^{x_0}_p(F)=\frac{I(F,G)}{\ext_{x_0}(G)^{1/2}}
\quad (F\in \mf).
\end{equation}
Indeed,
by definition,
there is a positive number $t_0$ such that
$\mathcal{E}_{[G]}(F)=t_0\,I(F,G)$ for all $F\in \mf$.
By \eqref{eq:minsky_inequality_sharp} and Convention \ref{convention:2},
we obtain
$$
1=\max_{F\in \mf_1}\mathcal{E}_{[G]}(F)
=t_0\max_{F\in \mf_1}I(F,G)
=t_0\ext_{x_0}(G)^{1/2}.
$$

The following is proven in \cite{Mi4}
by the similar argument as the case of the Thurston compactification
(cf. \cite{FLP}).

\begin{proposition} \label{prop:zero-boundary}
For $p\in\cl{\TeichbX}$,
the following are equivalent.
\begin{itemize}
\item[{\rm (1)}]
$p\in \partialGM{\Teich{g}{m}}$;
\item[{\rm (2)}]
there is an $F\in \mf-\{0\}$ with $\mathcal{E}_p(F)=0$.
\end{itemize}
\end{proposition}
%
%


\section{Cones $\GmInv$, $\GmTeich$ and $\Gmbdy$}
\label{sec:cones}

%
%
%
%

\subsection{Cones}
\label{subsec:cones}
Define
\begin{align}
\GmInv
&=\proj^{-1}(\cl{\TeichbX})\cup \{0\}\subset \mathbb{R}_+^{\mathcal{S}}
\label{eq:definition_GmInv} \\
\GmTeich
&=\proj^{-1}(\TeichbX)\cup \{0\}\subset \mathbb{R}_+^{\mathcal{S}}
\label{eq:definition_GmTeich} \\
\Gmbdy
&=
\proj^{-1}(\partialGM{\TeichbX})\cup \{0\}\subset \GmInv
\subset \mathbb{R}_+^{\mathcal{S}}.
\label{eq:definition_Gmbdy}
\end{align}
We topologize $\GmInv$, $\GmTeich$ and $\Gmbdy$
with the topology induced from $\mathbb{R}_+^{\mathcal{S}}$.
Notice that $\mf$
is contained in $ \Gmbdy$
as a closed subset
since $\pmf\subset \partialGM{\TeichbX}$.
In particular,
any $G\in \mf$
is nothing other than an assignment
\begin{equation}
\label{eq:embedded_MF}
\mathcal{S}\ni \alpha\mapsto I(\alpha,G).
\end{equation}

%

\subsection{Models of $\GmInv$, $\GmTeich$ and $\Gmbdy$}
\label{subsec:model_cones}
We define models of cones by
\begin{align*}
\GmInvC
&=\cl{\TeichbX}\times \mathbb{R}_+/(\cl{\TeichbX}\times \{0\}) \\
\GmTeichC
&=\TeichbX\times \mathbb{R}_+/(\TeichbX\times \{0\}) \\
\GmbdyC
&=\partialGM{\TeichbX}
\times \mathbb{R}_+/(\partialGM{\TeichbX}\times \{0\}) \\
\mfC
&=\pmf\times \mathbb{R}_+/(\pmf\times \{0\}) \\
\pmfC
&=\pmf\times \{1\}
\end{align*}
\begin{figure}
\begin{center}
\includegraphics[height=5cm]{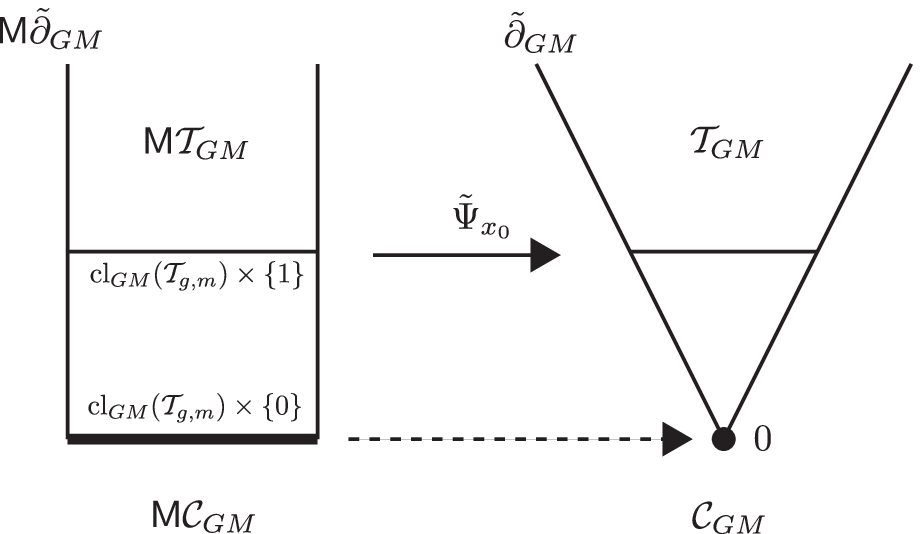}
\caption{Models and the model map of the cone $\GmInv$.}
\label{fig:Modelcone}
\end{center}
\end{figure}
(see Figure \ref{fig:Modelcone}).
Since $\pmf\subset \partialGM{\TeichbX}$,
$\pmfC\subset \mfC\subset \GmbdyC$.
In this setting,
we often identify $\cl{\TeichbX}$ with the slice
$\cl{\TeichbX}\times \{1\}$ of $\GmInvC$.

We abbreviate the point $(p,t)\in \GmInvC$ to $tp$.
We denote $1\,p$ by $p$ for the simplicity.
For $s\ge 0$ and $\zeta=tp\in \GmInvC$ with $t\ge 0$ and $p\in \cl{\TeichbX}$,
we define the multiplication $s\zeta$ by
$s\zeta=(st)p$.

From Proposition \ref{prop:continuity_E_p} and \eqref{eq:E_y},
the embedding \eqref{eq:Lift_GM} is continuous.
Therefore,
we have a continuous bijection (the \emph{model map})
$$
\tilde{\Psi}_{x_0}\colon
\GmInvC
\to \GmInv
$$
defined by
\begin{equation}
\label{eq:homeomorphism_cone}
\tilde{\Psi}_{x_0}(tp)=\tilde{\Psi}_{x_0}(p,t)=t\cdot \Psi_{x_0}(p)=[\mathcal{S}\ni \alpha\mapsto t\mathcal{E}_p(\alpha)].
\end{equation}
By definition,
$\tilde{\Psi}_{x_0}$ is homogeneous in the sense that
$$
\tilde{\Psi}_{x_0}(t\zeta)=t\tilde{\Psi}_{x_0}(\zeta)
$$
for $t\ge 0$ and $\zeta\in \GmInvC$
and satisfies $\tilde{\Psi}_{x_0}(p)=\Psi_{x_0}(p)$ for $p\in \cl{\TeichbX}$.
Since $\cl{\TeichbX}$ is compact,
the bijection $\tilde{\Psi}_{x_0}$ is a homeomorphism.
It follows from \eqref{eq:E-p-For-measured-foliation}
that
\begin{equation}
\label{eq:psi_mfC_to_MF}
\tilde{\Psi}_{x_0}(s[F])=s\,\ext_{x_0}(F)^{-1/2}\cdot F\in \mf
\end{equation}
for $s[F]\in \mfC$ and
hence $\tilde{\Psi}_{x_0}(\mfC)=\mf$.
In particular,
we deduce the following.

\begin{lemma}[Image of $\pmfC$]
\label{lem:image_of_PMF}
For $[G]\in \pmfC$,
we have
$\tilde{\Psi}_{x_0}([G])\in \mf_1\subset \Gmbdy$.
\end{lemma}


\begin{remark}
From the identification \eqref{eq:homeomorphism_cone},
we recognize $\GmInv$,
$\GmTeich$ and $\Gmbdy$
as cones with slices $\cl{\TeichbX}$,
$\TeichbX$ and $\partialGM{\TeichbX}$,
respectively.
\emph{Notice that this identification depends on the choice of the basepoint $x_0$
(cf. \eqref{eq:Lift_GM}}).
\end{remark}

\section{Intersection number and Extremal length associated to a basepoint}
\label{sec:extremal_length_GmInv}
In this section,
we define the intersection number  on $\GmInvC\times \mfC$ and
the extremal length for elements in $\GmInvC$ associated to
the basepoint $x_0$.
We will extend the intersection number given here to
the whole $\GmInvC\times \GmInvC$ in \S\ref{subsec:extension_i_x_0}.

\subsection{Intersection number associated to the basepoint}
\label{subsec:intersection_number_associated_to_basepoint}
For $\zeta=tp\in \GmInvC$ ($t\ge 0$ and $p\in \cl{\TeichbX}$)
and $\eta\in \mfC$,
we define the \emph{intersection number
associated to the basepoint $x_0$}
by
\begin{equation}
\label{eq:intersection_number_pre-definition}
i_{x_0}(\zeta,\eta)=i_{x_0}(tp,\eta)=
t\,\mathcal{E}_p
\left(\tilde{\Psi}_{x_0}(\eta)
\right)
=t\,\mathcal{E}^{x_0}_p
\left(\tilde{\Psi}_{x_0}(\eta)
\right).
\end{equation}
The intersection number \eqref{eq:intersection_number_pre-definition}
depends on the basepoint $x_0$.
Indeed,
By \eqref{eq:psi_mfC_to_MF},
we have
\begin{align}
\label{eq:intersection_teichmuller_measured_foliation}
i_{x_0}(ty,s[F])
&=t\,\mathcal{E}_y(\tilde{\Psi}_{x_0}(s[F]))
=t\,
\left\{\frac{\ext_y(s\ext_{x_0}(F)^{-1/2}\cdot F)}{K_y}
\right\}^{1/2} \nonumber \\
&=ts\cdot
e^{-d_T(x_0,y)}\,
\left(\frac{\ext_y(F)}{\ext_{x_0}(F)}
\right)^{1/2}
\nonumber
\end{align}
for $ty\in \GmTeichC$ and $s[F]\in \mfC$.
%
By \eqref{eq:homeomorphism_cone},
$\zeta\in \GmInvC$ corresponds to the function
\begin{equation} \label{eq:identification_GmInv}
\mathcal{S}\ni \alpha\mapsto i_{x_0}(\zeta,
\tilde{\Psi}_{x_0}^{-1}(\alpha))
\end{equation}
in $\GmInv$ via $\tilde{\Psi}_{x_0}$.
From Proposition \ref{prop:continuity_E_p},
the assignment
$$
\GmInvC\times \mfC\ni (\zeta,\eta)\mapsto i_{x_0}(\zeta,\eta)
$$
is continuous.
Furthermore,
the intersection number \eqref{eq:intersection_number_pre-definition}
is \emph{homogeneous}
since
\begin{align*}
i_{x_0}(s_1\zeta,s_2\eta)
&=i_{x_0}((s_1t)p,s_2\eta)=(s_1t)\mathcal{E}_p(\tilde{\Psi}_{x_0}(s_2\eta)) \\
&=s_1s_2\cdot t\mathcal{E}_p(\tilde{\Psi}_{x_0}(\eta)) =s_1s_2i_{x_0}(\zeta,\eta)
\end{align*}
where $s_1,s_2\ge 0$,
$\zeta=tp$ with $t\ge 0$ and $p\in \cl{\TeichbX}$,
and $\eta\in\mfC$.

\begin{proposition}[Intersection number on $\mf$]
\label{prop:intersection_MF}
The intersection number function
\eqref{eq:intersection_number_pre-definition}
coincides with the original intersection number
function on $\mf\times \mf$
via $\tilde{\Psi}_{x_0}$.
Namely,
when $G=\tilde{\Psi}_{x_0}(\zeta)$
and $F=\tilde{\Psi}_{x_0}(\eta)$
with $\zeta,\eta\in \mfC$,
$$
i_{x_0}(\zeta,\eta)=I(G,F).
$$
\end{proposition}

\begin{proof}
Notice from \eqref{eq:psi_mfC_to_MF}
that $\zeta=\tilde{\Psi}_{x_0}^{-1}(G)=\ext_{x_0}(G)^{1/2}\cdot [G]$.
By from \eqref{eq:E-p-For-measured-foliation},
we have
\begin{align*}
i_{x_0}(\zeta,\eta)
&=i_{x_0}(\ext_{x_0}(G)^{1/2}\cdot [G],\eta)
=\ext_{x_0}(G)^{1/2}\mathcal{E}_{[G]}(\tilde{\Psi}_{x_0}(\eta)) \\
&=\ext_{x_0}(G)^{1/2}\mathcal{E}_{[G]}(F)=I(G,F).
\hspace{\fill}\qed
\end{align*}
\end{proof}



\subsection{Extremal length on $\GmInvC$ associated to the basepoint}
\label{subsec:extremal_length_GmInv}
For $\zeta\in \GmInvC$,
we define the \emph{extremal length of $\zeta$
on $ty\in \GmTeichC$ associated to the basepoint $x_0$}
by
\begin{equation}
\label{eq:extremal_length_cone}
\extB^{x_0}_{ty}(\zeta)=
t^2\cdot\max_{\eta\in \pmfC}
\frac{i_{x_0}(\zeta ,\eta)^2}{\ext_y(\tilde{\Psi}_{x_0}(\eta))}
=
t^2\cdot\sup_{F\in \mfC-\{0\}}
\frac{i_{x_0}(\zeta ,\eta)^2}{\ext_y(\tilde{\Psi}_{x_0}(\eta))}.
\end{equation}
Then,
$\extB^{x_0}_{ty}(\,\cdot\,)$ is homogeneous and satisfies
\begin{equation}
\label{eq:minsky-inequality_mfC}
i_{x_0}(\zeta ,\eta)^2\le
\extB^{x_0}_{y}(\zeta)\cdot \ext_y(\tilde{\Psi}_{x_0}(\eta))
\end{equation}
for all
$y\in \TeichbX$,
$\zeta\in \GmInvC$ and $\eta\in \mfC$.
Since $\mfC_1$ is compact,
for every $\zeta\in \GmInvC$,
there is an $\eta\in \mfC-\{0\}$ such that
$$
\extB^{x_0}_{ty}(\zeta)=t^2\frac{i_{x_0}(\zeta ,\eta)^2}
{\ext_y(\tilde{\Psi}_{x_0}(\eta))}
$$
or
\begin{equation}
\label{eq:equality_definition_extremal_length_mfC}
i_{x_0}(\zeta ,\eta)^2=\extB^{x_0}_{ty}(\zeta)\cdot
\ext_y(\tilde{\Psi}_{x_0}(\eta)).
\end{equation}

\subsubsection{Basic properties}
We can easily see the following.

\begin{lemma}
\label{lem:measured_foliation_2}
The following two properties hold.
\begin{itemize}
\item[{\rm (1)}]
For
$ty,sz\in \GmTeichC$ with $t,s\ge 0$ and $y,z\in \TeichbX$,
$$
\extB^{x_0}_{ty}(sz)
=t^2s^2\,\exp(-2d_T(x_0,z)+2d_T(y,z)).
$$
\item[{\rm (2)}]
For $\zeta\in \mfC$
and $y\in \TeichbX$,
$$
\extB^{x_0}_{y}(\zeta)
=\ext_{y}(\tilde{\Psi}_{x_0}(\zeta)).
$$
\end{itemize}
\end{lemma}

\begin{proof}
(1)\quad
Since $K_z=\exp(2d_T(x_0,z))$,
from Kerckhoff's formula,
we have
\begin{align*}
\extB^{x_0}_{ty}(sz)
&=
t^2\cdot \sup_{\eta\in \mfC-\{0\}}
\frac{i_{x_0}(sz,\eta)^2}{\ext_y(\Psi_{x_0}(\eta))}
=
t^2s^2\,
\sup_{F\in \mf-\{0\}}
\frac{\ext_z(F)}{K_z\ext_y(F)} \\
&=t^2s^2\,\exp(-2d_T(x_0,z)+2d_T(y,z)).
\end{align*}

\medskip
\noindent
(2)\quad
This follows form Proposition \ref{prop:intersection_MF} and
\eqref{eq:tautological_formula}.
\qed
\end{proof}

We notice the following non-triviality of the extremal length
\eqref{eq:extremal_length_cone}.

\begin{lemma}[Non-triviality]
\label{lem:extremal_nontrivial}
Let $\zeta\in \GmInvC$.
If $\extB^{x_0}_{y}(\zeta)=0$ for some $y\in \TeichbX$,
then $\zeta=0$.
\end{lemma}

\begin{proof}
Take $t\ge 0$ and $p\in \cl{\TeichbX}$
with $\zeta=tp$.
Suppose $\extB^{x_0}_y(\zeta)=0$.
From \eqref{eq:intersection_number_pre-definition},
we have
\begin{align*}
0
=\extB^{x_0}_{y}(\zeta)
&=\sup_{\eta\in \mfC-\{0\}}
\frac{i_{x_0}(\zeta,\eta)^2}{\ext_y(\tilde{\Psi}_{x_0}(\eta))}
=\sup_{\eta\in \mfC-\{0\}}
\frac{\mathcal{E}_\zeta(\tilde{\Psi}_{x_0}(\eta))^2}{\ext_y(\tilde{\Psi}_{x_0}(\eta))} \\
&=\sup_{F\in \mf-\{0\}}
\frac{\mathcal{E}_\zeta(F)^2}{\ext_y(F)} 
=\sup_{F\in \mf-\{0\}}
\frac{t^2\,\mathcal{E}_{p}(F)^2}{\ext_y(F)}.
\end{align*}
Therefore,
we obtain
$$
t\,\mathcal{E}_{p}(F)=0
$$
for all $F\in \mf-\{0\}$.
On the other hand,
since $p\in \cl{\TeichbX}$,
$\mathcal{E}_{p}(\alpha)\ne 0$ for some $\alpha\in \mathcal{S}$,
and we get $t=0$.
Therefore,
$\zeta=tp=0$.
\qed
\end{proof}

\subsubsection{Continuity}
Notice that
the extremal length given in \eqref{eq:extremal_length_cone}
satisfies the \emph{distortion property}:
\begin{equation}
\label{eq:distortion_property_extension}
e^{-2d_T(y_1,y_2)}
\extB^{x_0}_{y_1}(\zeta)
\le
\extB^{x_0}_{y_2}(\zeta)
\le
e^{2d_T(y_1,y_2)}
\extB^{x_0}_{y_1}(\zeta)
\end{equation}
for $y_1,y_2\in \TeichbX$ and $\zeta\in \GmInvC$.
Indeed,
%
since $\tilde{\Psi}_{x_0}(\eta)\in \mf$ for $\eta\in \mfC$,
we have
$$
\ext_{y_1}(\tilde{\Psi}_{x_0}(\eta))\ge
e^{-2d_T(y_1,y_2)}\ext_{y_2}(\tilde{\Psi}_{x_0}(\eta))
$$
for all $\eta\in \mfC$.
Therefore,
we obtain
\begin{align*}
\extB^{x_0}_{y_2}(\zeta)
&=
\sup_{\eta\in \mfC-\{0\}}
\frac{i_{x_0}(\zeta ,\eta)^2}{\ext_{y_2}(\tilde{\Psi}_{x_0}(\eta))} \le 
e^{2d_T(y_1,y_2)}
\sup_{\eta\in \mfC-\{0\}}
\frac{i_{x_0}(\zeta ,\eta)^2}{\ext_{y_1}(\tilde{\Psi}_{x_0}(\eta))} \\
&=e^{2d_T(y_1,y_2)}\extB^{x_0}_{y_1}(\zeta).
\end{align*}

The following lemma immediately
follows
from Proposition \ref{prop:continuity_E_p}
and the above observation,
and we omit the proof.

\begin{lemma}[Continuity]
\label{lem:continuity}
The function
\begin{equation}
\label{eq:extremal_length_continuity}
\GmTeichC\times \GmInvC\ni (ty,\zeta)\mapsto \extB^{x_0}_{ty}(\zeta)
\end{equation}
is continuous.
\end{lemma}
%

\subsection{Extremal length is intrinsic}
\label{subsec:extremal_length_is_intrinsic}
The extremal length \eqref{eq:extremal_length_cone}
is \emph{intrinsic} in the following sense.

\begin{theorem}[Extremal length is intrinsic]
\label{thm:intrinsic_extremal_length}
For $y\in \TeichbX$,
there is a continuous function
$$
\ext_y:\GmInv\to \mathbb{R}_{+}
$$
such that
\begin{itemize}
\item[{\rm (1)}]
$\extB^{x}_y(\zeta)=\ext_y\circ \tilde{\Psi}_{x}(\zeta)$
for $\zeta\in \GmInvC$ and $x\in \TeichbX$,
and
\item[{\rm (2)}]
For $F\in \mf\subset \GmInv$,
the value $\ext_y(F)$ is equal to 
the original extremal length of $F$.
\end{itemize}
\end{theorem}

\begin{remark}
From the property {\rm (2)} in
Theorem \ref{thm:intrinsic_extremal_length},
the extremal length obtained in Theorem \ref{thm:intrinsic_extremal_length}
is a continuous extension of the original extremal length on $\mf$.
Thus,
the author believes that no confusion occurs when we use the same symbol to denote the extension of the extremal length in Theorem \ref{thm:intrinsic_extremal_length}.
\end{remark}

\begin{proof}[Proof of Theorem \ref{thm:intrinsic_extremal_length}]
We only check the existence and the property (1)
because the property (2) follows from Lemma \ref{lem:measured_foliation_2}.

Let $t,s>0$ and $x_1,x_2,z,w\in \TeichbX$.
Suppose that
$\tilde{\Psi}_{x_1}(tz)=\tilde{\Psi}_{x_2}(sw)$.
Then,
$$
te^{-d_T(x_1,z)}\ext_z(\alpha)^{1/2}=se^{-d_T(x_2,w)}\ext_w(\alpha)^{1/2}
$$
for all $\alpha\in \mathcal{S}$.
From the injectivity of the Gardiner-Masur embedding \eqref{eq:GM-embedding}
we have $z=w$
(cf. Lemma 6.1 in \cite{GM}).
Hence
\begin{equation}
t=s\exp(d_T(x_1,z)-d_T(x_2,z)).
\end{equation}
By Lemma \ref{lem:measured_foliation_2},
we obtain
\begin{align*}
\extB^{x_1}_y(tz)
&=t^2\exp(-2d_T(x_1,z)+2d_T(y,z))\\
&=s^2\exp(2d_T(x_1,z)-2d_T(x_2,z))\cdot \exp(-2d_T(x_1,z)+2d_T(y,z)) \\
&=s^2\exp(-2d_T(x_2,z)+2d_T(y,z)) \\
&=\extB^{x_2}_y(sz)=\extB^{x_2}_y(sw).
\end{align*}
Therefore,
there is a function $\ext_y:\GmTeich\to \mathbb{R}$
such that
\begin{equation}
\label{eq:intrinsic_extremal_on_Teich}
\ext_y(\mathfrak{a})=
\extB^{x_0}_y\circ (\tilde{\Psi}_{x_0})^{-1}(\mathfrak{a}).
\end{equation}
for all $\mathfrak{a}\in \GmTeich$.
From the continuity of $\extB^{x_0}_y$
on $\GmInvC$ and $\tilde{\Psi}_{x_0}^{-1}$
on $\GmInv$,
the function
$\ext_y$ in 
\eqref{eq:intrinsic_extremal_on_Teich}
extends to whole $\GmInv$,
and
\eqref{eq:intrinsic_extremal_on_Teich} holds for all $\mathfrak{a}\in \GmInv$.
%
%
\qed
\end{proof}

\section{Topology of the model}
\label{sec:topology_of_models}
Notice that $\cl{\TeichbX}$ is separable and metrizable
(cf. \cite{Mi1}).
Hence,
$\GmInvC$ and $\GmInv$ are locally compact,
separable and metrizable.

\subsection{Bounded sets are precompact}
We shall begin with the following proposition.

\begin{proposition}[Boundedness implies compactness]
\label{prop:bdd_implies_compact}
For any $R>0$,
$$
\GmInvC(R)=\{\zeta\in \GmInvC\mid \extB^{x_0}_{x_0}(\zeta)\le R\}
$$
is a compact set in $\GmInvC$.
Furthermore,
the level set
$$
\{\zeta\in \GmInvC\mid \extB^{x_0}_{x_0}(\zeta)=1\}
$$
coincides with $\cl{\TeichbX}\times \{1\}$.
In particular $\extB^{x_0}_{x_0}(\zeta)=1$ for $\zeta\in \pmfC$.
\end{proposition}

\begin{proof}
From the definition \eqref{eq:extremal_length_cone},
the condition $\extB^{x_0}_{x_0}(\zeta)\le R$
implies that
$$
i_{x_0}(\zeta,\tilde{\Psi}_{x_0}^{-1}(\alpha))\le R^{1/2}\ext_{x_0}(\alpha)^{1/2}
$$
for all $\alpha\in \mathcal{S}$.
By Tikhonov's theorem,
the product of closed intervals
$$
\prod_{\alpha\in \mathcal{S}}[0,R^{1/2}\ext_{x_0}(\alpha)^{1/2}]
$$
is a compact set in $\mathbb{R}_+^{\mathcal{S}}$.
From \eqref{eq:identification_GmInv},
the image of $\GmInvC(R)$ by $\tilde{\Psi}_{x_0}$
is contained in the above product.
Thus,
by Lemma \ref{lem:continuity},
$\GmInvC(R)$ is closed and hence compact.
The second claim immediately follows from the first and
Lemma \ref{lem:measured_foliation_2}.
%
%
\qed
\end{proof}

\subsection{A system of neighborhoods}
\label{subsec:system_neighborhood}
Let $(\zeta,\xi)\in \GmInvC\times  \mfC$ with $\zeta,\xi\ne 0$ and $\delta>0$.
We define
\begin{align*}
U_\delta(\zeta\colon\xi)
&=\{\eta\in \GmInvC\mid
|i_{x_0}(\eta,\xi)-i_{x_0}(\zeta,\xi)|<
\extB^{x_0}_{x_0}(\zeta)^{1/2}\extB^{x_0}_{x_0}(\xi)^{1/2}\delta\} \\
U_\delta(0\colon\xi)
&=\{\eta\in \GmInvC\mid
i_{x_0}(\eta,\xi)<\extB^{x_0}_{x_0}(\xi)^{1/2}\delta\}.
\end{align*}
Notice that
\begin{equation}
\label{eq:neighborhood-homogeneity}
U_\delta(\zeta\colon t\xi)=U_\delta(\zeta\colon\xi)
\end{equation}
for $t>0$ and $(\zeta,\xi)\in \GmInvC\times  \mfC$ with $\xi\ne 0$.
We set
$$
U_\delta
(\zeta)=\cap_{\xi\in \mfC-\{0\}}U_\delta(\zeta:\xi).
$$
We start with the following lemma.

\begin{proposition}
\label{prop:relaton_neighborhood}	
Let $\delta>0$ and $\zeta\in \GmInvC$.
Then
$$
(1-\delta)\extB^{x_0}_{x_0}(\zeta)^{1/2}
<
\extB^{x_0}_{x_0}(\eta)^{1/2}
<
(1+\delta)\extB^{x_0}_{x_0}(\zeta)^{1/2}
$$
for $\eta\in U_\delta(\zeta)$.
\end{proposition}

\begin{proof}
From \eqref{eq:equality_definition_extremal_length_mfC}
and
Proposition \ref{prop:bdd_implies_compact},
we can find $\xi\in \pmfC$
such that
$$
i_{x_0}(\zeta,\xi)^2=
\extB^{x_0}_{x_0}(\zeta)\cdot \extB^{x_0}_{x_0}(\xi)
=\extB^{x_0}_{x_0}(\zeta).
$$
By \eqref{eq:minsky-inequality_mfC},
for $\eta\in U_\delta(\zeta)$,
we have
$$
\extB^{x_0}_{x_0}(\zeta)^{1/2}
=i_{x_0}(\zeta,\xi)
<
i_{x_0}(\eta,\xi)+
\extB^{x_0}_{x_0}(\zeta)^{1/2}\delta
\le
\extB^{x_0}_{x_0}(\eta)^{1/2}+
\extB^{x_0}_{x_0}(\zeta)^{1/2}\delta
$$
and hence 
$$
(1-\delta)\extB^{x_0}_{x_0}(\zeta)^{1/2}\le \extB^{x_0}_{x_0}(\eta)^{1/2}.
$$
Similarly,
we take $\xi\in \pmfC$
with 
$i(\eta,\xi)^2=
\extB^{x_0}_{x_0}(\eta)\extB^{x_0}_{x_0}(\xi)=\extB^{x_0}_{x_0}(\eta)$.
This means that
$$
\extB^{x_0}_{x_0}(\eta)^{1/2}
=i_{x_0}(\eta,\xi)
<
i_{x_0}(\zeta,\xi)+
\extB^{x_0}_{x_0}(\zeta)^{1/2}\delta
\le
\extB^{x_0}_{x_0}(\zeta)^{1/2}+
\extB^{x_0}_{x_0}(\zeta)^{1/2}\delta,
$$
and we are done.
\qed
\end{proof}

We claim the following (compare Lemma 4.1 of \cite{Mi1}. See also \cite{Ker}).

\begin{lemma}
Let $\zeta\in \GmInvC$.
For any $\delta>0$,
$U_\delta(\zeta)$ is an open neighborhood of $\zeta$
with compact closure.
Furthermore,
we have that $\cap_{\delta>0}U_\delta(\zeta)=\{\zeta\}$.
\end{lemma}

\begin{proof}
It is clear that $\zeta\in U_\delta(\zeta)$ for all $\delta>0$.
%
Let $\zeta'\in U_{\delta}(\zeta)$.
We suppose on the contrary that 
there is a sequence $\{\zeta_n\}_{n=1}^\infty$ in the complement
$\GmInvC\setminus U_\delta(\zeta)$
which converges to $\zeta'$.
For any $n$,
there is $\xi_n\in \pmfC$ such that
\begin{equation} \label{eq:contradiction_open_U_delta}
|i_{x_0}(\zeta_n,\xi_n)-i_{x_0}(\zeta,\xi_n)|
\ge \extB^{x_0}_{x_0}(\zeta)^{1/2}\ext_{x_0}(\xi_n)^{1/2}\delta
=\extB^{x_0}_{x_0}(\zeta)^{1/2}\delta.
\end{equation}
Since $\pmfC$ is compact,
we may assume that $\xi_n$ converges to $\xi_\infty\in \pmfC$.
Since $\zeta_n\to \zeta'$ as $n\to \infty$,
by Proposition \ref{prop:continuity_E_p}
and \eqref{eq:contradiction_open_U_delta},
we have
$$
|i_{x_0}(\zeta',\xi_\infty)-i_{x_0}(\zeta,\xi_\infty)|
\ge \extB^{x_0}_{x_0}(\zeta)^{1/2}\delta,
$$
and we get a contradiction by Lemma \ref{lem:extremal_nontrivial}.
Hence $U_\delta(\zeta)$ is open.
By Lemma \ref{prop:relaton_neighborhood},
$U_\delta(\zeta)$
is contained in $\GmInvC((1+\delta)\extB^{x_0}_{x_0}(\zeta))$.
Therefore,
by Proposition \ref{prop:bdd_implies_compact},
the closure of $U_\delta(\zeta)$ is compact.

To show the remaining claim,
we only treat the case $\zeta\ne 0$.
The other case is dealt with the same manner.
Suppose that $\eta\in U_\delta(\zeta)$ for all $\delta>0$.
By definition,
we have
$$
|i_{x_0}(\eta,\xi)-i_{x_0}(\zeta,\xi)|<\extB^{x_0}_{x_0}(\zeta)^{1/2}\delta
$$
 for all $\xi\in \pmfC$ and $\delta>0$.
 This means that $i_{x_0}(\eta,\xi)=i_{x_0}(\zeta,\xi)$
 for all $\xi\in \pmfC$ and $\eta=\zeta$.
\qed
\end{proof}

%
%
%

\section{The Gromov product and Extension of $\mathcal{E}_\zeta$}
\label{sec:Gromov_product_extension}
For $\eta=ty\in \GmTeichC$
and $\zeta\in \GmInvC$,
we define
\begin{equation}
\label{eq:E_y-2}
\mathcal{E}_{\eta}(\zeta)=
\left\{
\frac{\extB^{x_0}_{ty}(\zeta)}{K_y}
\right\}^{1/2}
=
t\cdot \exp(-d_T(x_0,y))\cdot
\extB^{x_0}_{y}(\zeta)^{1/2}.
\end{equation}
After identifying $\mfC$ and $\mf$ via $\tilde{\Psi}_{x_0}$,
by Lemma \ref{lem:measured_foliation_2},
the function $\mathcal{E}_y$ in \eqref{eq:E_y-2}
is recognized as an extension of the function \eqref{eq:E_y} to $\GmInvC$.
By definition,
the function \eqref{eq:E_y-2}
satisfies the \emph{homogeneous property}
\begin{equation}
\label{eq:bilinearity}
\mathcal{E}_{sy}(t\zeta)=st\cdot \left\{
\frac{\extB^{x_0}_{y}(\zeta)}{K_y}
\right\}^{1/2}
=st\cdot \mathcal{E}_y(\zeta).
\end{equation}
for $sy\in \GmTeichC$,
$t\ge 0$ and $\zeta\in \GmInvC$.

Notice from Lemma \ref{lem:measured_foliation_2}
that
\begin{align}
\label{eq:gromov-product_extension}
\mathcal{E}_{sy}(tz)
&=st\cdot\exp(-d_T(x_0,z)+d_T(y,z)-d_T(x_0,y))\\
&=st\cdot \exp(-2\gromov{y}{z}{x_0}) \nonumber
\end{align}
for $sy,tz\in \GmTeichC$
where $\gromov{y}{z}{x_0}$
is the \emph{Gromov product}
$$
\gromov{y}{z}{x_0}=\frac{1}{2}
(d_T(x_0,z)+d_T(x_0,y)-d_T(y,z))
$$
with basepoint $x_0$.
In particular,
we have the following \emph{symmetry}
\begin{equation}
\label{eq:symmetry_E_p}
\mathcal{E}_{sy}(tz)
=\mathcal{E}_{tz}(sy) \\
\end{equation}
for $sy,tz\in \GmTeichC$.

The following  was observed for the extremal length function on $\mf$
in \cite{Mi1}.

\begin{proposition}[Equicontinuity]
\label{prop:equicontinuity}
The family $\{\mathcal{E}_y\}_{y\in \TeichbX}$
is an equicontinuous family of continuous functions on $\GmInvC$.
In fact,
for $\delta>0$ and $\zeta\in \GmInvC$,
we have
\begin{equation}
\label{eq:equicontinuity_extension}
|\mathcal{E}_y(\zeta)-\mathcal{E}_y(\eta)|
\le \max\{1,\extB^{x_0}_{x_0}(\zeta)^{1/2}\}\delta
\end{equation}
for all $\eta\in U_\delta(\zeta)$ and $y\in \TeichbX$.
\end{proposition}

\begin{proof}
We first assume that $\zeta\ne 0$.
Take $\xi\in \pmfC$ with
$i_{x_0}(\zeta,\xi)=\extB^{x_0}_y(\zeta)^{1/2}\extB^{x_0}_y(\xi)^{1/2}$
(cf. \eqref{eq:equality_definition_extremal_length_mfC}).
If $\eta\in U_\delta(\zeta)$,
\begin{align*}
\extB^{x_0}_y(\zeta)^{1/2}
\extB^{x_0}_y(\xi)^{1/2}
&=i_{x_0}(\zeta,\xi)
\le i_{x_0}(\eta,\xi) +\extB^{x_0}_{x_0}(\zeta)^{1/2}\delta\\
&\le
\extB^{x_0}_y(\eta)^{1/2}\extB^{x_0}_y(\xi)^{1/2}+
\extB^{x_0}_{x_0}(\zeta)^{1/2}\delta.
\end{align*}
Hence we get
\begin{align}
\extB^{x_0}_y(\zeta)^{1/2}
&\le \extB^{x_0}_y(\eta)^{1/2}
+\frac{\extB^{x_0}_{x_0}(\zeta)^{1/2}}{\extB^{x_0}_y(\xi)^{1/2}}\delta \nonumber\\
&\le \extB^{x_0}_y(\eta)^{1/2}
+K_y^{1/2}\extB^{x_0}_{x_0}(\zeta)\delta,
\label{eq:extremal_C_GM_1}
\end{align}
since $\extB^{x_0}_y(\xi)\ge K_y^{-1}\extB^{x_0}_{x_0}(\xi)=K_y^{-1}$
(cf. \eqref{eq:distortion_property_extension}).

We also take $\xi'\in \pmfC$
with $i_{x_0}(\eta,\xi')=\extB^{x_0}_y(\eta)^{1/2}\extB^{x_0}_y(\xi')^{1/2}$.
Then,
\begin{align*}
\extB^{x_0}_y(\eta)^{1/2}\extB^{x_0}_y(\xi')^{1/2}
&=
i_{x_0}(\eta,\xi')
\le
i_{x_0}(\zeta,\xi')+\extB^{x_0}_{x_0}(\zeta)\delta\\
&\le
\extB^{x_0}_y(\zeta)^{1/2}\extB^{x_0}_y(\xi')^{1/2}+
\extB^{x_0}_{x_0}(\zeta)^{1/2}\delta.
\end{align*}
Hence,
by the same argument as above,
\begin{equation}
\label{eq:extremal_C_GM_2}
\extB^{x_0}_y(\eta)^{1/2}
\le
\extB^{x_0}_y(\zeta)^{1/2}+K_y^{1/2}\extB^{x_0}_{x_0}(\zeta)^{1/2}\delta.
\end{equation}
Thus,
\eqref{eq:extremal_C_GM_1} and \eqref{eq:extremal_C_GM_2}
\begin{equation}
\label{eq:extremal_C_GM_3}
|\extB^{x_0}_y(\eta)^{1/2}-\extB^{x_0}_y(\zeta)^{1/2}|\le K_y^{1/2}\extB^{x_0}_{x_0}(\zeta)^{1/2}\delta.
\end{equation}

Suppose next that $\zeta=0$.
If we take $\xi'\in \pmfC$
with $i_{x_0}(\eta,\xi')=\extB^{x_0}_y(\xi')^{1/2}\extB^{x_0}_y(\eta)^{1/2}$,
$$
\extB^{x_0}_y(\eta)^{1/2}\cdot \extB^{x_0}_y(\xi')^{1/2}
=i_{x_0}(\eta,\xi')<\delta.
$$
Therefore,
we conclude
\begin{equation}
\label{eq:extremal_C_GM_4}
|\extB^{x_0}_y(\eta)^{1/2}-\extB^{x_0}_y(\zeta)^{1/2}|=
\extB^{x_0}_{y}(\eta)\le
\frac{\delta}{\extB^{x_0}_y(\xi')^{1/2}}
\le K_y^{1/2}\delta.
\end{equation}
Thus,
\eqref{eq:extremal_C_GM_3} and \eqref{eq:extremal_C_GM_4}
implies \eqref{eq:equicontinuity_extension}.
\qed
\end{proof}


\section{Extension of the intersection number}
\label{sec:extension_intersection_number}

One of the purpose of this section is to show the following theorem.

\begin{theorem}[Intersection number on $\GmInv$]
\label{thm:pairing_intrinsic}
There exists a unique continuous function
\begin{equation}
\label{eq:new-intersection-number}
i(\cdot,\cdot):\GmInv\times \GmInv\to \mathbb{R}_+
\end{equation}
independent of the choice of basepoint $x_0$
satisfying the following properties.
\begin{itemize}
\item[{\rm (1)}]
For any $\zeta,\eta\in \GmInvC$,
$$
i(\tilde{\Psi}_{x_0}(\zeta),\tilde{\Psi}_{x_0}(\eta))
=i_{x_0}(\zeta,\eta).
$$
In particular,
we have
\begin{align*}
i\left(\tilde{\Psi}_{x_0}(ty),\tilde{\Psi}_{x_0}(sp)\right)
&=
ts\,e^{-d_T(x_0,y)}\ext_y\left(\Psi_{x_0}(p)\right)^{1/2} \\
i\left(\tilde{\Psi}_{x_0}(p),F\right)
&=
i\left(\Psi_{x_0}(p),F\right)=
\mathcal{E}_p(F)
\end{align*}
for $x_0,y\in \TeichbX$, $p\in \partialGM{\TeichbX}$,
$F\in \mf$ and $t,s\ge 0$.
\item[{\rm (2)}]
$i(\mathfrak{a},\mathfrak{b})=i(\mathfrak{b},\mathfrak{a})$
for $\mathfrak{a},\mathfrak{b}\in \GmInv$.
\item[{\rm (3)}]
$i(s\mathfrak{a},t\mathfrak{b})=st\cdot i(\mathfrak{a},\mathfrak{b})$
for $s,t\ge 0$ and $\mathfrak{a},\mathfrak{b}\in \GmInv$.
\item[{\rm (4)}]
For $x_0,y,z\in \TeichbX$,
$$
i\left(\Psi_{x_0}(y),\Psi_{x_0}(z)\right)=\exp(-2\gromov{y}{z}{x_0}).
$$
\item[{\rm (5)}]
The self-intersection number satisfies
$$
i(\mathfrak{a},\mathfrak{a})=
\left\{
\begin{array}{cl}
t^2\exp(-2d_T(x_0,y)) &
\mbox{if $\mathfrak{a}=\tilde{\Psi}_{x_0}(ty)\in \GmTeich$} \\
0 & \mbox{if $\mathfrak{a}\in \Gmbdy$}
\end{array}
\right.
$$
for $x_0\in \TeichbX$.
\item[{\rm (6)}]
For $F,G\in \mf\subset \GmInv$,
$$
i(F,G)=I(F,G),
$$
where
we recall that the intersection number
in the right-hand side is the original intersection number function on $\mf\times \mf$.
\end{itemize}
\end{theorem}

Theorem \ref{thm:main_realization_0} follows from Theorem \ref{thm:pairing_intrinsic}.
Indeed,
the only difference is the item (iv) in Theorem \ref{thm:main_realization_0}:
From (3) in Theorem \ref{thm:pairing_intrinsic},
we have
\begin{align}
i(\tilde{\Phi}_{GM}(y),\tilde{\Phi}_{GM}(z))
&=\exp(d_T(x_0,y))\cdot \exp(d_T(x_0,z))
\cdot
i(\Psi_{x_0}(y),\Psi_{x_0}(z))
\label{eq:exchange}\\
&=
\exp(d_T(y,z))
\nonumber 
\end{align}
for $y,z\in \TeichbX$.

\subsubsection*{Corollaries}
We give two corollaries of Theorem \ref{thm:pairing_intrinsic}
before proving the theorem.


\begin{corollary}[Minsky's inequality]
\label{coro:minsky_inequality_extention}
For $x\in \TeichbX$ and $\mathfrak{a},\mathfrak{b}\in \GmInv$,
we have
\begin{equation}
\label{eq:minsky_inequality_assertion}
i(\mathfrak{a},\mathfrak{b})^2\le
\ext_x(\mathfrak{a})\,\ext_x(\mathfrak{b}).
\end{equation}
The equality holds if the projective classes
of $\mathfrak{a}, x$ and $\mathfrak{b}$ are on a common Teichm\"uller geodesic
in this order.
\end{corollary}

\begin{proof}
Suppose that $\mathfrak{a},\mathfrak{b}\in \GmTeich$.
Take $ty,sz\in \GmTeichC$
with $\mathfrak{a}=\tilde{\Psi}_{x_0}(ty)$
and $\mathfrak{b}=\tilde{\Psi}_{x_0}(sz)$.
Then,
by Lemma \ref{lem:measured_foliation_2},
we have
\begin{align}
i(\mathfrak{a},\mathfrak{b})^2
&=
i_{x_0}(\tilde{\Psi}_{x_0}(ty),\tilde{\Psi}_{x_0}(sz))^2
=t^2s^2\exp(-4\gromov{y}{z}{x_0}) \nonumber \\
&=t^2s^2\exp(2d_T(y,z)-2d_T(x_0,y)-2d_T(x_0,z)) \nonumber \\
&\le t^2s^2\exp(2d_T(x,y)-2d_T(x_0,y))
\cdot
\exp(2d_T(x,z)-2d_T(x_0,z)) \label{eq:minsky-inequality_sharp}\\
&
=\extB^{x_0}_{x}(ty)\cdot \extB^{x_0}_{x}(sz)
=\ext_x(\mathfrak{a})\cdot \ext_x(\mathfrak{b}).
\nonumber 
\end{align}
Since $\GmTeich$ is dense in $\GmInv$,
we have the desired inequality.

Suppose the projective classes of
$\mathfrak{a}, x$ and $\mathfrak{b}$ are on
a common Teichm\"uller geodesic $\gamma:\mathbb{R}\to \TeichbX$
in this order.
We may assume that $\mathfrak{a},\mathfrak{b}\in \partialGM{\TeichbX}$
since intersection number and extremal length are homogeneous.
From the assumption,
we may choose $\gamma$ such that
$\gamma(t)\to \mathfrak{a}$ and $\gamma(-t)\to \mathfrak{b}$
when $t\to \infty$.
Therefore,
from \eqref{eq:minsky-inequality_sharp}
we have
$$
i(\gamma(t),\gamma(-t))^2=\ext_x(\gamma(t))\cdot \ext_x(\gamma(-t))
$$
for sufficiently large $t>0$.
By letting $t\to \infty$,
we get the equality in \eqref{eq:minsky_inequality_assertion}.
\qed
\end{proof}

\begin{corollary}[Intrinsic representation of extremal length]
For $y\in \TeichbX$ and $\mathfrak{a}\in \GmInv$,
we have
\begin{equation}
\label{eq:extremal_representation_intrinsic}
\ext_y(\mathfrak{a})=
\sup_{F\in \mf-\{0\}}
\frac{i(\mathfrak{a},F)^2}{\ext_y(F)}
=\sup_{\mathfrak{b}\in \GmInv-\{0\}}
\frac{i(\mathfrak{a},\mathfrak{b})^2}{\ext_y(\mathfrak{b})}.
\end{equation}
\end{corollary}

\begin{proof}
Notice that in the definition \eqref{eq:extremal_length_cone}
of the extremal length,
the measured foliation $F$ in the denominator in \eqref{eq:extremal_length_cone}
is taken in $\mfC-\{0\}\subset \GmInvC$.
Therefore,
by (2) of Theorem \ref{thm:intrinsic_extremal_length}
and Theorem \ref{thm:pairing_intrinsic},
for $\mathfrak{a}\in \GmInv$,
we have
$$
\ext_y(\mathfrak{a})
=
\extB^{x_0}_y\circ \tilde{\Psi}_{x_0}^{-1}(\mathfrak{a})
=
\sup_{F\in \mf-\{0\}}
\frac{i_{x_0}\left(
\tilde{\Psi}_{x_0}^{-1}(\mathfrak{a}),
\tilde{\Psi}_{x_0}^{-1}(F)
\right)^2}
{\ext_y(F)}=
\sup_{F\in \mf-\{0\}}
\frac{i\left(
\mathfrak{a},
F
\right)^2}
{\ext_y(F)}.
$$
The second equality follows from Corollary \ref{coro:minsky_inequality_extention}.
\qed
\end{proof}

\subsection{Extension of the intersection number $i_{x_0}$}
\label{subsec:extension_i_x_0}
To show Theorem \ref{thm:pairing_intrinsic},
we first extend the intersection number
\eqref{eq:intersection_number_pre-definition}
to the whole $\GmInvC\times \GmInvC$.

%
%

\begin{proposition}[Extension of $i_{x_0}$]
\label{prop:pairing}
For any $x_0\in \TeichbX$,
there exists a unique continuous function
\begin{equation}
\label{eq:new-intersection-number_basepoint}
i_{x_0}(\cdot,\cdot):\GmInvC\times \GmInvC\to \mathbb{R}_+
\end{equation}
such that
\begin{itemize}
\item[{\rm (1)}]
For $ty\in \GmTeichC$ and $sp\in \GmbdyC$
with $y\in \TeichbX$, $p\in \partialGM{\TeichbX}$ and $t,s\ge 0$,
$$
i_{x_0}(ty,sp)=ts\,\mathcal{E}_y(p)=ts\,e^{-d_T(x_0,y)}\extB^{x_0}_y(p)^{1/2};
$$
\item[{\rm (2)}]
$i_{x_0}(\zeta,\eta)=i_{x_0}(\eta,\zeta)$
for $\zeta,\eta\in \GmInvC$;
\item[{\rm (3)}]
$i_{x_0}(s\zeta,t\eta)=st\cdot i_{x_0}(\zeta,\eta)$
for $s,t\ge 0$ and $\zeta,\eta\in \GmInvC$;
\item[{\rm (4)}]
for $y,z\in \TeichbX$, $i_{x_0}(y,z)=\exp(-2\gromov{y}{z}{x_0})$;
\item[{\rm (5)}]
for $\zeta=tp\in \GmInvC$ with $p\in \cl{\TeichbX}$;
$$
i_{x_0}(\zeta,\zeta)=
\left\{
\begin{array}{cl}
t^2\exp(-2d_T(x_0,p)) & \mbox{if $\zeta\in \GmTeichC$} \\
0 & \mbox{if $\zeta\in \GmbdyC$};
\end{array}
\right.
$$
\item[{\rm (6)}]
$i_{x_0}(\tilde{\Psi}_{x_0}^{-1}(F),\tilde{\Psi}_{x_0}^{-1}(G))=I(F,G)$
for all $F,G\in \mf$.
\end{itemize}
\end{proposition}

\begin{proof}
Consider the equicontinuous family $\{\mathcal{E}_y\}_{y\in \TeichbX}$
given in Proposition \ref{prop:equicontinuity}.
For any $\zeta\in \GmInvC$,
$$
\mathcal{E}_y(\zeta)=
\left\{
\frac{\extB^{x_0}_y(\zeta)}{K_y}
\right\}^{1/2}
\le \extB^{x_0}_{x_0}(\zeta)^{1/2}.
$$
By Proposition \ref{prop:bdd_implies_compact},
the family $\{\mathcal{E}_y\}_{y\in \TeichbX}$ is uniformly bounded on any compact set.
Therefore,
the family is a normal family.

Let $\zeta\in \GmInvC$.
Let $p\in \cl{\TeichbX}$ and $t\ge 0$ such that $\zeta=tp$.
Let $\{y_n\}_{n=1}^\infty$ be a sequence converging to $p$.
Take a sequence $\{t_n\}_{n=1}^\infty$ of positive numbers with $t_n\to t$.
By Ascoli-Arzel\`a theorem,
there is a subsequence $\{y_{n_j}\}_j$ such that
a sequence $\{\mathcal{E}_{t_{n_j}y_{n_j}}\}_j$ converges
to the continuous function $\mathcal{E}'$ on $\GmInvC$
uniformly on any compact set.
Notice from Lemma \ref{lem:continuity} and \eqref{eq:symmetry_E_p}
that for $sz\in \GmTeichC$,
\begin{equation}
\label{eq:limit_on_GmTeich}
\mathcal{E}'(sz)=
\lim_{j\to \infty}
\mathcal{E}_{t_{n_j}y_{n_j}}(sz)
=
\lim_{j\to \infty}
\mathcal{E}_{sz}(t_{n_j}y_{n_j})
=\mathcal{E}_{sz}(\zeta).
\end{equation}
Take another sequence $\{t'_ky'_k\}_k$ 
in $\GmTeichC$ which tends to $\zeta$
such that $\mathcal{E}_{t_ky'_k}$
converges to a continuous function $\mathcal{E}''$ on $\GmInvC$
uniformly on any compact set of $\GmInvC$.
Since the right-hand side of \eqref{eq:limit_on_GmTeich}
is independent of converging sequences,
the same conclusion holds for $\mathcal{E}''$.
Namely,
we have
$$
\mathcal{E}''(sz)=\mathcal{E}_{sz}(\zeta)=\mathcal{E}'(sz)
$$
for all $sz\in \GmTeichC$.
Since $\GmTeichC$ is dense in $\GmInvC$ and both $\mathcal{E}''$ and $\mathcal{E}'$
are continuous on $\GmInvC$,
$\mathcal{E}''=\mathcal{E}'$ on $\GmInvC$.
This means that the limit $\mathcal{E}'$ above is dependent only on $\zeta$,
independent of the choice of the sequence $\{y_n\}_{n=1}^\infty$
converging to $\zeta$.
We denote by $i_{x_0}(\zeta,\cdot )$ the limit.

For any $R>0$,
notice again that
$\{\mathcal{E}_{sy}\}_{sy\in \GmInvC(R)}$ is a normal family
of continuous functions on $\GmInvC$.
From the argument above,
\begin{equation}
\label{eq:Gminv-GmInv}
\GmInvC\times \GmInvC\ni (\zeta,\eta)\mapsto i_{x_0}(\zeta,\eta)
\end{equation}
is continuous in two variables.
The condition (1) in the statement follows from the construction 
and \eqref{eq:E_y-2}.
From the 
density of $\GmTeichC\times \GmInvC$ in $\GmInvC\times \GmInvC$
we deduce the uniqueness of our function $i_{x_0}(\cdot,\cdot)$.

Let us check that our function $i_{x_0}(\cdot,\cdot)$
satisfies the remaining conditions (2) to (5) in the statement.
Indeed,
(2) and (3) follows from 
the density of $\GmTeichC$ in $\GmInvC$
and equations \eqref{eq:bilinearity} and \eqref{eq:symmetry_E_p}.
We get (4) from \eqref{eq:gromov-product_extension}.
The condition (5) is verified from
$$
i_{x_0}(\zeta,\zeta)=t^2\exp(-2\gromov{y}{y}{x_0})
=t^2\exp(-d_T(x_0,y))
$$
when $\zeta=ty\in \GmTeichC$
and the continuity of the function $i_{x_0}(\cdot,\cdot)$.
The last condition (6) follows from Proposition \ref{prop:intersection_MF}.
\qed
\end{proof}
%
\paragraph{{\bf Proof of Theorem \ref{thm:pairing_intrinsic}.}}
Theorem \ref{thm:pairing_intrinsic}
immediately follows from Proposition \ref{prop:pairing}.
Indeed,
we define
\begin{equation}
\label{eq:intersection_number_intrinsic}
i(\mathfrak{a},\mathfrak{b})
=i_{x_0}(\tilde{\Psi}_{x_0}^{-1}(\mathfrak{a}),\tilde{\Psi}_{x_0}^{-1}(\mathfrak{b}))
\end{equation}
for $\mathfrak{a},\mathfrak{b}\in \GmInv$.
By applying the similar argument as that in the proof of
Theorem \ref{thm:intrinsic_extremal_length},
one can see
that the intersection
number \eqref{eq:intersection_number_intrinsic} is intrinsic in the sense that
the value is independent of the choice of basepoint
$x_0$.
\qed

\subsection{Extension of the Gromov product}
\label{subsec:extension_Gromov}
In this section,
we give a proof of Corollary \ref{coro:Gromov_product_boundary}.
%
The uniqueness of the extension follows from the density of $\TeichbX$ in $\cl{\TeichbX}$
and the condition (1) in Corollary \ref{coro:Gromov_product_boundary}.
Hence it suffices to show the existence.

Define
\begin{equation} \label{eq:new-pairing}
\gromov{p}{q}{x_0}
=-\frac{1}{2}\log i_{x_0}(p,q)
\end{equation}
for $p,q\in \cl{\TeichbX}$,
where $\cl{\TeichbX}$ is identified with
a subset via the embedding \eqref{eq:Lift_GM}.
Notice from Proposition \ref{prop:bdd_implies_compact} and
Corollary \ref{coro:minsky_inequality_extention}
that $i_{x_0}(p,q)\le 1$ for $p,q\in \cl{\TeichbX}$.
Therefore,
the pairing $\gromov{\cdot}{\cdot}{x_0}$
defined above is continuous with value in $[0,\infty]$.
From (4) of Proposition \ref{prop:pairing},
the pairing \eqref{eq:new-pairing}
coincides with the Gromov product with basepoint $x_{0}$.
Since
\begin{align*}
I(F,G)&=i_{x_0}(\tilde{\Psi}_{x_0}^{-1}(F),\tilde{\Psi}_{x_0}^{-1}(G))
=i_{x_0}(\ext_{x_0}(F)^{1/2}[F],\ext_{x_0}(G)^{1/2}[G]) \\
&=\ext_{x_0}(F)^{1/2}\cdot \ext_{x_0}(G)^{1/2} i_{x_0}([F],[G]),
\end{align*}
we conclude (2) of Corollary \ref{coro:Gromov_product_boundary}.
\qed

\section{Isometric action on Teichm\"uller space}
\label{sec:isometric_action}
An orientation preserving
homeomorphism $h\colon X\to X$ induces 
a homeomorphic action $h_*$
on $\partialGM{\Teich{g}{m}}$ by the equation
\begin{equation}
\label{eq:extension_homeo_on_X}
\mathcal{E}_{h_*(p)}(F)=t\,\mathcal{E}_p(h^{-1}(F))
\end{equation}
for all $F\in \mf$,
where $t>0$ is independent of $F$.
Indeed,
the action $h_*$ is the homeomorphic extension of
the Teichm\"uller modular group action on $\Teich{g}{m}$ induced by $h$
(cf. \S5.4 of \cite{Mi1}).
In this section,
we give a necessary condition for a mapping of $\partialGM{\Teich{g}{m}}$
to be induced from a homeomorphism on $X$.

\subsection{Null space}
\label{subsec:null_space}
For $\mathfrak{a}\in \GmInv$,
we define the \emph{null space of $\mathfrak{a}$}
by
$$
\mathcal{N}(\mathfrak{a})=
\{\mathfrak{b}\in \GmInv\mid i(\mathfrak{a},\mathfrak{b})=0\}.
$$
By definition,
$0\in \mathcal{N}(\mathfrak{a})$ for all $\mathfrak{a}\in \GmInv$.
We remark the following simple claim.

\begin{proposition}
\label{prop:translation_null_space1}
The following hold.
\begin{itemize}
\item[(1)]
For $\mathfrak{a}\in \GmInv-\{0\}$,
$\mathcal{N}(\mathfrak{a})\ne \{0\}$
if and only if $\mathfrak{a}\in \Gmbdy$.
\item[(2)]
$\mathcal{N}(\mathfrak{a})\subset \Gmbdy$
for all $\mathfrak{a}\in \GmInv$.
\item[(3)]
$\mathcal{N}(\mathfrak{a})\cap \mf\ne \{0\}$
for $\mathfrak{a}\in \Gmbdy$.
\end{itemize}
\end{proposition}

\begin{proof}
(1)\quad
If $\mathfrak{a}\in \GmTeich$,
$\mathcal{N}(\mathfrak{a})=\{0\}$
from Lemma \ref{lem:extremal_nontrivial} and
(1) of Theorem \ref{thm:pairing_intrinsic}.
If $\mathfrak{a}\in \Gmbdy$,
from (5) of Theorem \ref{thm:pairing_intrinsic},
$\mathfrak{a}\in \mathcal{N}(\mathfrak{a})$
and $\mathcal{N}(\mathfrak{a})\ne \{0\}$.

\medskip
(2)\quad
From (1) above,
$\mathcal{N}(\mathfrak{a})=\{0\}\subset \Gmbdy$ for $\mathfrak{a}\in \GmTeich$.
Let $\mathfrak{a}\in \Gmbdy$.
For any $\mathfrak{b}\in \mathcal{N}(\mathfrak{a})$,
$\mathfrak{a}\in \mathcal{N}(\mathfrak{b})\ne \{0\}$.
This means that $\mathcal{N}(\mathfrak{a})\cap \GmTeich=\{0\}$
for all $\mathfrak{a}\in \GmInv$.

\medskip
(3)\quad
Let $\mathfrak{a}\in \Gmbdy$.
Suppose $\mathfrak{a}=\tilde{\Psi}_{GM}(tp)$ for some $t\ge 0$
and $p\in \partialGM{\TeichbX}$.
If $\mathcal{N}(\mathfrak{a})\cap \mf=\{0\}$,
$$
t\mathcal{E}_p(F)=i(\tilde{\Psi}_{x_0}(tp),\tilde{\Psi}_{x_0}(F))
=i(\mathfrak{a},F)\ne 0
$$
for all $F\in \mf-\{0\}$
by
Theorem \ref{thm:pairing_intrinsic}.
By Proposition \ref{prop:zero-boundary},
this implies $p\in \TeichbX$,
which is a contradiction.
\qed
\end{proof}

Let $\omega$ be a mapping $\omega\colon \cl{\Teich{g}{m}}\to \cl{\Teich{g}{m}}$.
We extend the action of $\omega$ to $\GmInvC$ by
$$
H_\omega:\GmInvC\ni tp\mapsto t\,\omega(p)\in \GmInvC
$$
where $t\ge 0$ and $p\in \cl{\Teich{g}{m}}$.
Let $x_0\in \TeichbX$ be the basepoint as before.
We define a homeomorphism $h_\omega$ on $\GmInv$ by
$$
h_\omega=\tilde{\Psi}_{x_0}\circ H_\omega\circ \tilde{\Psi}_{x_0}^{-1}.
$$

\begin{proposition}
\label{prop:translation_null_space2}
Let $\omega\colon:\Teich{g}{m}\to \Teich{g}{m}$ be a mapping of bounded distortion
for triangles.
Suppose that $\omega$ admits a continuous extension to $\cl{\Teich{g}{m}}$.
Then,
for $\mathfrak{a},\mathfrak{b}\in \Gmbdy$,
$i(h_\omega(\mathfrak{a}),h_\omega(\mathfrak{b}))=0$
if and only if $i(\mathfrak{a},\mathfrak{b})=0$.
Furthermore,
if $\omega$ has a quasi-inverse $\omega'$
which also admits a continuous extension to $\cl{\Teich{g}{m}}$,
then
\begin{equation}
\label{eq:translation_null_space2}
h_{\omega'}\circ h_{\omega}(\mathcal{N}(\mathfrak{a}))
\subset \mathcal{N}(\mathfrak{a})
\end{equation}
when $\mathfrak{a}\in \Gmbdy$.
\end{proposition}

\begin{proof}
Let $D_1$ and $D_2$ be the distortion constants of $\omega$.
%
A formal calculation yields
\begin{align*}
2\gromov{\omega(y)}{\omega(z)}{\omega(x_0)}
&=
2\gromov{\omega(y)}{\omega(z)}{x_0}
-
2\gromov{\omega(x_{0})}{\omega(y)}{x_0} \\
&
\qquad
-2\gromov{\omega(x_{0})}{\omega(z)}{x_0}
-
2d_{T}(x_{0},\omega(x_{0}))
%
\end{align*}
for every $x,y\in \TeichbX$.
Since $\omega$ is a mapping of bounded distortion for triangles
with constant $D_1,D_2>0$,
we conclude that
\begin{equation} \label{eq:intersection_mcg_1}
e^{-2D_2}J_{x_0}(y,z)\,i_{x_0}(y,z)^{D_1}
\le i_{x_0}(\omega(y),\omega(z))
\le
e^{2D_2}J_{x_0}(y,z)\,i_{x_0}(y,z)^{\frac{1}{D_1}},
\end{equation}
where
$$
J_{x_0}(y,z)=
e^{2d_T(x_0,\omega(x_0))}
i_{x_0}(\omega(x_0),\omega(y))\,
i_{x_0}(\omega(x_0),\omega(z)).
$$

Let $\zeta,\eta\in \partialGM{\Teich{g}{m}}$.
Since $\omega$ has a continuous extension to $\cl{\Teich{g}{m}}$,
by letting $y\to \zeta$ and $z\to \eta$ in \eqref{eq:intersection_mcg_1},
we get
\begin{equation}
\label{eq:intersection_mcg_2}
e^{-2D_2}J_{x_0}(\zeta,\eta)\,i_{x_0}(\zeta,\eta)^{D_1}
\le i_{x_0}(\omega(\zeta),\omega(\eta))
\le
e^{2D_2}J_{x_0}(\zeta,\eta)\,i_{x_0}(\zeta,\eta)^{\frac{1}{D_1}}
\end{equation}
from Proposition \ref{prop:pairing},
where
\begin{align*}
J_{x_0}(\zeta,\eta)
&=
\lim_{y\to \zeta,z\to \eta}
e^{2d_T(x_0,\omega(x_0))}
i_{x_0}(\omega(x_0),\omega(y))\,
i_{x_0}(\omega(x_0),\omega(z)) \\
&=
e^{2d_T(x_0,\omega(x_0))}
\extB^{x_0}_{\omega(x_0)}(\omega(\zeta))^{1/2}
\extB^{x_0}_{\omega(x_0)}(\omega(\eta))^{1/2}
\ne 0
\end{align*}
since $\omega(x_0)\in \Teich{g}{m}$ (cf. Lemma \ref{lem:extremal_nontrivial}).
Therefore,
\eqref{eq:intersection_mcg_2} implies that
$i_{x_0}(\omega(\zeta),\omega(\eta))=0$
if and only if $i_{x_0}(\zeta,\eta)=0$
for $\zeta,\eta\in \partialGM{\Teich{g}{m}}$.

Let $\mathfrak{a},\mathfrak{b}\in \Gmbdy$.
Take $\zeta,\eta\in \partialGM{\Teich{g}{m}}$ and $t,s\ge 0$ with
$\mathfrak{a}=\tilde{\Psi}_{x_0}(t\zeta)$
and $\mathfrak{b}=\tilde{\Psi}_{x_0}(s\eta)$.
Then,
by \eqref{eq:intersection_number_intrinsic},
\begin{align*}
i(\mathfrak{a},\mathfrak{b})
&=i_{x_0}(t\zeta,s\eta)
=ts\,
i_{x_0}(\zeta,\eta) \\
i(h_\omega(\mathfrak{a}),h_\omega(\mathfrak{b}))
&=i_{x_0}(H_\omega(t\zeta),H_\omega(s\eta))
=ts\,
i_{x_0}(\omega(\zeta),\omega(\eta)).
\end{align*}
Therefore,
$i(\mathfrak{a},\mathfrak{b})=0$
if and only if 
$i(h_\omega(\mathfrak{a}),h_\omega(\mathfrak{b}))=0$.

%

Suppose $\omega$ has a quasi-inverse $\omega'$ of
quasi-inverse constant $D_3$
which extends continuously to $\cl{\Teich{g}{m}}$.
Then,
$$
2\gromov{y}{z}{x_0}-2D_3
\le
2\gromov{y}{\omega'\circ \omega(z)}{x_0}
\le 
2\gromov{y}{z}{x_0}+2D_3
$$
and
$$
e^{-2D_3}i_{x_0}(y,z)\le
i_{x_0}(y,\omega'\circ \omega(z))
\le
e^{2D_3}i_{x_0}(y,z).
$$
Therefore,
by letting $y\to \zeta$ and $z\to \eta$,
we have
$$
e^{-2D_3}i_{x_0}(\zeta,\eta)\le
i_{x_0}(\zeta,\omega'\circ \omega(\eta))
\le
e^{2D_3}i_{x_0}(\zeta,\eta)
$$
for all $\zeta,\eta\in \partialGM{\Teich{g}{m}}$,
which implies
$$
e^{-2D_3}i(\mathfrak{a},\mathfrak{b})\le
i(\mathfrak{a},h_{\omega'}\circ h_\omega(\mathfrak{b}))
\le
e^{2D_3}i(\mathfrak{a},\mathfrak{b})
$$
for $\mathfrak{a},\mathfrak{b}\in \Gmbdy$.
Let $\mathfrak{b}\in h_{\omega'}\circ h_{\omega}(\mathcal{N}(\mathfrak{a}))$.
Take $\mathfrak{c}\in \mathcal{N}(\mathfrak{a})$
with $\mathfrak{b}=h_{\omega'}\circ h_{\omega}(\mathfrak{c})$.
Since
$$
i(\mathfrak{a},\mathfrak{b})
=i(\mathfrak{a},h_{\omega'}\circ h_{\omega}(\mathfrak{c}))
\le e^{2D_3}i(\mathfrak{a},\mathfrak{c})=0,
$$
we have
$\mathfrak{b}\in \mathcal{N}(\mathfrak{a})$.
\qed
\end{proof}

\subsection{$\omega$ preserves $\pmf$}
\label{subsec:omega_preserves_pmf}
This section is devoted to showing (1) in Theorem \ref{thm:almost_isometry}.
Namely,
we prove the following.

\begin{proposition}[$\omega$ preserves $\pmf$]
\label{prop:OMEGApreservePMF}
Let $\omega\colon \Teich{g}{m}\to \Teich{g}{m}$ be a mapping of bounded distortion for triangles
with continuous extension to $\cl{\Teich{g}{m}}$.
Suppose that $\omega$ has a quasi-inverse $\omega'$ which also extends
continuously to $\cl{\Teich{g}{m}}$.
Then,
the restriction of $\omega$ to $\pmf$ is a self-homeomorphism
of $\pmf$.
Furthermore,
$\omega'=\omega^{-1}$ on $\pmf$.
\end{proposition}

The proof of Proposition \ref{prop:OMEGApreservePMF}
is given in
\S\ref{subsubsec:Proof_of_homeomorphism}.
Before showing Proposition \ref{prop:OMEGApreservePMF},
we deal with uniquely ergodic measured foliations as elements in $\GmInv$
in the next section.

\subsubsection{Uniquely ergodic measured foliations}
\label{subsec:ue_in_GmInv}
In this paper,
$G\in \mf-\{0\}$ is said to be \emph{uniquely ergodic}
if 
every $F\in (\mathcal{N}(G)-\{0\})\cap \mf$ is projectively equivalent to $G$.
It is known that the set of uniquely ergodic measured foliations 
is dense in $\mf$
(cf. \cite{FLP}. See also \cite{Masur2} and \cite{Veech}).

In the Gardiner-Masur boundary,
simple closed curves and
uniquely ergodic measured foliations are rigid
in the following sense.

\begin{lemma}[Theorem 3 of \cite{Mi3}]
\label{lem:charcterization_ue}
Let $p\in \cl{\TeichbX}$.
Let $G\in \mf$ be a simple closed curve or a uniquely ergodic measured foliation.
Suppose that $\mathcal{E}_p(F)=0$
for all $F\in \mathcal{N}(G)\cap \mf$.
Then there is $t>0$ such that
$$
\mathcal{E}_p(F)=t\,i(F,G)
$$
for all $F\in \mf$.
Namely, $p=[G]$ as points in $\cl{\TeichbX}$.
\end{lemma}

We give a characterization
of uniquely ergodic measured foliations
as follows.

\begin{lemma}[Uniquely ergodic points]
\label{lem:unuquely_ergodic_points}
The following four conditions are equivalent
for $\mathfrak{a}\in \GmInv-\{0\}$:
\begin{itemize}
\item[{\rm (i)}]
There exists $\mathfrak{b}\in \GmInv-\{0\}$ such that
$\mathcal{N}(\mathfrak{a})=\{t\mathfrak{b}\mid t\ge 0\}$.
\item[{\rm (ii)}]
$\mathcal{N}(\mathfrak{a})=\{t\mathfrak{a}\mid t\ge 0\}$.
\item[{\rm (iii)}]
$\mathfrak{a}\in \mf$ and $\mathfrak{a}$ is uniquely ergodic.
\item[{\rm (iv)}]
$\mathcal{N}(\mathfrak{a})$ contains a uniquely ergodic measured foliation.
\end{itemize}
\end{lemma}

\begin{proof}
{\bf (i) is equivalent to (ii).}\quad
Clearly (ii) implies (i).
Since $\mathcal{N}(\mathfrak{a})\ne \{0\}$,
$\mathfrak{a}\in \Gmbdy$.
Thus,
(ii) follows from (i) since $i(\mathfrak{a},\mathfrak{a})=0$
(cf. Theorem \ref{thm:pairing_intrinsic}).

\medskip
\noindent
{\bf (ii) implies (iii).}\quad
By (1) and (3) of Proposition \ref{prop:translation_null_space1},
$\mathfrak{a}\in \Gmbdy$ and $\mathcal{N}(\mathfrak{a})\cap \mf\ne \{0\}$.
Therefore, 
we have $\mathfrak{a}\in \mf$.
Thus,
if $F\in \mf$ satisfies $I(F,\mathfrak{a})=0$,
$F$ is projectively equivalent to $\mathfrak{a}$.
This means that $\mathfrak{a}$ is a uniquely ergodic measured foliation.

\medskip
\noindent
{\bf (iii) implies (ii).}\quad
Let $G\in \mf\subset \GmInv$ be
a uniquely ergodic measured foliation.
Let $\mathfrak{b}\in \mathcal{N}(G)-\{0\}$.
From Proposition \ref{prop:translation_null_space1},
$\mathfrak{b}\in \Gmbdy$.
Let $p\in \partialGM{\TeichbX}$ and $t>0$ with $\mathfrak{b}=\tilde{\Psi}_{x_0}(tp)$.
Then,
by
Theorem \ref{thm:pairing_intrinsic},
$$
t\,\mathcal{E}_p(G)=i(\tilde{\Psi}_{x_0}(tp),G)=i(\mathfrak{b},G)=0.
$$
By Lemma \ref{lem:charcterization_ue},
$\mathfrak{b}$ is projectively equivalent to $G$.
This means that $\mathcal{N}(G)=\{tG\mid t\ge 0\}$.

\medskip
\noindent
{\bf (iii) is equivalent to (iv).}\quad
Clearly (iii) implies (iv) since $\mathfrak{a}\in \mathcal{N}(\mathfrak{a})$.
Suppose $\mathcal{N}(\mathfrak{a})$ contains
a uniquely ergodic measured foliation $G$.
Since $i(\mathfrak{a},G)=0$,
by applying the same argument  in ``(iii) implies (ii)" above,
we deduce $\mathfrak{a}$ is projectively equivalent to $G$
and $\mathfrak{a}$ is a uniquely ergodic measured foliation.
\qed
\end{proof}

\subsubsection{Proof of Proposition \ref{prop:OMEGApreservePMF}}
\label{subsubsec:Proof_of_homeomorphism}
Let $G\in \mf\subset \GmInv$ be
a uniquely ergodic measured foliation.
Since $\mathcal{N}(G)=\{tG\mid t\ge 0\}$,
we have from Proposition \ref{prop:translation_null_space2}
that
$$
 h_{\omega'}\circ  h_{\omega}(\mathcal{N}(G))
 \subset \mathcal{N}(G)=\{tG\mid t\ge 0\}.
 $$
Since $h_{\omega'}\circ  h_{\omega}(G)\in h_{\omega'}\circ  h_{\omega}(\mathcal{N}(G))$,
$h_{\omega'}\circ  h_{\omega}(\mathcal{N}(G))\ne \{0\}$.
Therefore,
$$
h_{\omega'}\circ  h_{\omega}(\mathcal{N}(G))=\mathcal{N}(G)=\{tG\mid t\ge 0\}.
$$
This implies that $\omega'\circ \omega([G])=[G]$.

Since
the set $\pmf^{UE}$ of uniquely ergodic measured foliations
is dense in $\pmf$
and $\omega$ and $\omega'$ are continuous,
we conclude that $\omega'\circ \omega$ is the identity mapping on $\pmf$.
By applying the same argument,
we deduce that $\omega\circ \omega'$ is also
the identity on $\pmf$.
In particular,
since
$$
\pmf
=\omega\circ \omega'(\pmf)
\subset \omega(\partialGM{\TeichbX}),
$$
$\mf$ is contained in both $h_\omega(\Gmbdy)$
and $h_{\omega'}(\Gmbdy)$.

Let $[G]\in \pmf^{UE}$ again.
By Proposition \ref{prop:translation_null_space1},
we can take $F\in \mathcal{N}(h_\omega(G))\cap\mf$
with $F\ne 0$.
Since $\mf\subset h_\omega(\Gmbdy)$,
there is an $\mathfrak{a}\in \Gmbdy$
such that $F=h_\omega(\mathfrak{a})$.
Since $i(h_\omega(\mathfrak{a}),h_\omega(G))=i(F,h_\omega(G))=0$,
we have
from Proposition \ref{prop:translation_null_space2} that
$i(\mathfrak{a},G)=0$.
Hence,
it follows from Lemma \ref{lem:unuquely_ergodic_points}
that $\mathfrak{a}=tG$
for some $t>0$.
Therefore $h_\omega(G)=t^{-1}F\in \mf$,
and
$\omega([G])\in \pmf$
for all $[G]\in \pmf^{UE}$.
By applying the same argument to $h_{\omega'}$,
we conclude that $\omega(\pmf)\subset \pmf$
and $\omega'(\pmf)\subset \pmf$
from the density of uniquely ergodic measured foliations in $\mf$.

On the other hand,
since
$\omega\circ \omega'$ and $\omega'\circ \omega$ are the identity on $\pmf$,
we deduce
$$
\pmf
=\omega\circ \omega'(\pmf)
\subset \omega(\pmf)
\subset \pmf
$$
and we are done.
\qed

\subsubsection{Null space in $\mf$}
From Proposition \ref{prop:OMEGApreservePMF},
we have the following observation.

\begin{proposition}
\label{prop:omega_preserving_mf}
Let $\omega$ be as Proposition \ref{prop:OMEGApreservePMF}.
For $G\in \mf$,
$$
h_\omega(\mathcal{N}(G)\cap \mf)=
\mathcal{N}(h_\omega(G))\cap \mf.
$$
\end{proposition}

\begin{proof}
Take a quasi-inverse $\omega'$ of $\omega$.
Notice as in Proposition \ref{prop:OMEGApreservePMF}
that $\omega'=\omega^{-1}$
on $\pmf$.
Therefore,
the restrictions of
$h_\omega$ and $h_{\omega'}$ to $\mf$ are
self-homeomorphisms of $\mf$
and
$h_{\omega}=h_{\omega'}^{-1}$ on $\mf$.

Take $F\in \mathcal{N}(h_\omega(G))\cap \mf$.
Since $i(h_\omega\circ h_{\omega'}(F),h_\omega(G))=i(F,h_\omega(G))=0$,
we have $i(h_{\omega'}(F),G)=0$
and $h_{\omega'}(F)\in \mathcal{N}(G)\cap \mf$
from Proposition \ref{prop:translation_null_space2}.
Therefore,
\begin{align*}
\mathcal{N}(h_\omega(G))\cap \mf
&\subset
h_{\omega'}^{-1}(\mathcal{N}(G)\cap \mf) \\
&\subset
h_{\omega'}^{-1}(\mathcal{N}(G))
\cap h_{\omega'}^{-1}(\mf)
=
h_\omega(\mathcal{N}(G))\cap \mf.
\end{align*}
Conversely,
let $F\in h_\omega(\mathcal{N}(G))\cap \mf$.
Take $H\in \mathcal{N}(G)$ with $h_\omega(H)=F$.
By Proposition \ref{prop:translation_null_space2} again,
$i(h_{\omega'}(F),G)=i(H,G)=0$
implies $i(F,h_\omega(G))=0$.
Therefore,
we obtain
$F\in \mathcal{N}(h_\omega(G))\cap \mf$
and
$$
h_\omega(\mathcal{N}(G))\cap \mf
\subset
\mathcal{N}(h_\omega(G))\cap \mf,
$$
and we are done.
\qed
\end{proof}


\subsection{Proof of Theorem \ref{thm:almost_isometry}}
\label{subsec:proof_of_theorem_characterization}
From Proposition \ref{prop:OMEGApreservePMF},
it suffices to check the assertion (2) in the theorem.


We identify $\alpha\in \mathcal{S}$ as an element of $\Gmbdy$ by
\eqref{eq:scc_embedded}.
Then,
by Proposition \ref{prop:OMEGApreservePMF},
$h_\omega(\alpha)\in \mf$.
Notice that
$\mathcal{N}(\alpha)\cap \mf$ is a subset of
codimension one in $\mf$.
By Proposition \ref{prop:omega_preserving_mf},
so is $\mathcal{N}(h_\omega(\alpha))\cap \mf$
since $h_\omega$ is a self-homeomorphism of $\mf$.
Since the complex dimension of $\Teich{g}{m}$ is at least $2$,
by virtue of Theorem 4.1 in \cite{Ivanov},
we deduce that $h_\omega(\alpha)\in \mathbb{R}_+\otimes \mathcal{S}$.
By applying the same argument to the quasi-inverse $\omega'$,
we conclude that the action of $\omega$ on $\pmf$ preserves $\mathcal{S}$.
Namely,
$\omega$ is a bijection from $\mathcal{S}$ onto $\mathcal{S}$.

Let $\alpha$, $\beta\in \mathcal{S}$ with $i(\alpha,\beta)=0$.
Then,
$\beta\in \mathcal{N}(\alpha)\cap \mf$.
By the argument above,
$h_\omega(\beta)\in \mathcal{N}(h_\omega(\alpha))\cap\mf$
and hence $i(h_\omega(\alpha),h_\omega(\beta))=0$.
This means that $\omega\colon \mathcal{S}\to \mathcal{S}$
induces an automorphism of the complex of curves of $X$.
\qed

\subsection{Proof of Corollary \ref{coro:Ivanov_characterization}}
\label{subsec:isometry_on_TX}
The purpose of this section
is to prove
Corollary \ref{coro:Ivanov_characterization} by applying Theorem \ref{thm:almost_isometry}.
It is known that
any isometry of $(\Teich{g}{m},d_T)$
extends to $\partialGM{\Teich{g}{m}}$ as a homeomorphism
(cf. \cite{LiuSu}. See also \cite{BriF}, \cite{Gromov}  and \cite{Rieffel}).

\subsubsection{Action of the extended mapping class group}
Before proving Corollary \ref{coro:Ivanov_characterization},
we shall recall the action of the extended mapping class group
on Teichm\"uller space (cf. \cite{IT} and \cite{MW}).

The \emph{extended mapping class group} $\mcg^*(X)$ is defined by
$$
\mcg^*(X)={\rm Diff}(X)/{\rm Diff}_0(X)
$$
where ${\rm Diff}(X)$ is the group of diffeomorphisms of $X$
and ${\rm Diff}_0(X)$ is the normal subgroup of ${\rm Diff}(X)$
consisting of diffeomorphisms which are isotopic to the identity.
Here,
we may choose $X$ so that it admits an antiholomorphic reflection $j_X\colon
X\to X$.

Let $\psi\in {\rm Diff}(X)$.
If $\psi$ is an orientation preserving diffeomorphism,
the action of the mapping class of $\psi$ is defined by
$$
\psi_*(Y,f)=(Y,f\circ \psi^{-1}).
$$
If $\psi$ is represented by an orientation reversing diffeomorphism,
there is an orientation preserving diffeomorphism $\vartheta_\psi$ such that
$\psi$ is isotopic to $\vartheta_\psi\circ j_X$.
Then,
the action of $\psi$ is defined by
$$
\psi_*(Y,f)=(Y^*,\overline{r}_Y\circ f\circ j_X\circ \vartheta_\psi^{-1}),
$$
where $Y^*$ is the conjugate Riemann surface to $Y$,
that is,
the coordinate charts of $Y^*$ are those of $Y$
followed by complex conjugations,
and $\overline{r}_Y:Y\to Y^*$ is the anticonformal mapping
induced by the identity mapping on the underlying surface of $Y$.

The following is well-known.
However,
we give a proof here because
the author cannot find a suitable reference in the case of
the action of orientation reversing diffeomorphisms.

\begin{lemma}[Isometry]
\label{lem:isometric_action_extended_mapping_class_group}
Any element in the extended mapping class group
acts isometrically on $(\Teich{g}{m},d_T)$.
\end{lemma}

\begin{proof}
Let $\psi\in \mcg^*(X)$.
If $\psi$ is represented by an orientation preserving diffeomorphism,
the assertion is well-known (cf. e.g \cite{IT}).

Suppose that $\psi$ is represented by an orientation reversing diffeomorphism.
Let $\vartheta_\psi$ as above.
From the original definition of the Teichm\"uller distance
\eqref{eq:original_teichmuller_distance},
we have
$$
d_T(\psi_*(Y_1,f_1),\psi_*(Y_2,f_2))=\frac{1}{2}
\log \inf_{h'}K(h')
$$
where $h'$ which runs over all quasiconformal mapping from $Y_1^*$
to $Y_2^*$
homotopic to
$$
(f_2\circ j_X\circ \vartheta_\psi)\circ (f_1\circ j_X\circ \vartheta_\psi)^{-1}
=\overline{r}_{Y_2}\circ f_2\circ f_1^{-1}\circ \overline{r}_{Y_1}^{-1}.
$$
Since each $\overline{r}_{Y_i}$ are anticonformal,
the action of $\psi_*$ is an isometry.
\qed
\end{proof}

In the proof of the following lemma,
we use the following simple formula:
For any simple closed curve $\alpha$ on a Riemann surface $Y$,
\begin{equation} \label{eq:extremal_length_conjugation}
\ext_{Y^*}(\overline{r}_{Y}(\alpha))=\ext_Y(\alpha).
\end{equation}
Indeed,
the modulus of an annulus does not change
under taking the complex conjugation
(cf. \eqref{eq:extremal_length_original}).

\begin{lemma}[Action at the boundary]
\label{lem:action_bdy}
For $\psi\in \mcg^*(X)$,
the restriction of the action of $\psi$ to
$\pmf\subset \partialGM{\Teich{g}{m}}$
coincides with the canonical action of $\psi$ on $\pmf$,
that is,
the continuous extension of the action 
$\mathcal{S}\ni \alpha\mapsto \psi(\alpha)\in \mathcal{S}$.
\end{lemma}

\begin{proof}
Let $\psi\in \mcg^*(X)$.
We only check the case where
$\psi$ corresponds to an orientation reversing diffeomorphism.
The other case can be treated in a similar way
(cf. e.g. Theorem 1.3 of \cite{Mi1}).

For $\alpha\in \mathcal{S}$,
we denote by $R_{\alpha,y}\colon [0,\infty)\to \Teich{g}{m}$
the Teichm\"uller geodesic ray
which  emanates from $y$ and is defined by the Jenkins-Strebel differential
on $y$ whose vertical foliation is $\alpha$.
Let $(X_t,f_t)=R_{\alpha,x_0}(t)$ for $t\ge 0$.
Let $p_\infty\in \partialGM{\Teich{g}{m}}$
be the limit of the Teichm\"uller geodesic ray $t\mapsto \psi_*(R_{\alpha,x_0}(t))$.

Take $\beta\in \mathcal{S}$ with $i(\alpha,\beta)=0$.
From the proof of Theorem 5.1 of \cite{GM},
$$
\ext_{X_t}(f_t(\beta))=\ext_{R_{\alpha,x_0}(t)}(\beta)=O(1)
$$
as $t\to \infty$
(see also \cite{Ker}).
Take $\vartheta_\psi$ as above.
Since $\vartheta_\psi\circ j_X$ is isotopic to $\psi$,
\begin{align}
\ext_{\psi_*(R_{\alpha,x_0}(t))}(\psi(\beta))
&=\ext_{X_t^*}(\overline{r}_{X_t}\circ f_t\circ j_X\circ \vartheta_\psi^{-1}(\psi(\beta))) \nonumber \\
&=\ext_{X_t^*}(\overline{r}_{X_t}\circ f_t(\beta))=\ext_{X_t}(f_t(\beta))=O(1)
\label{eq:action_psi_beta}
\end{align}
as $t\to \infty$ (cf. \eqref{eq:extremal_length_conjugation}).
This means that the corresponding function
$\mathcal{E}_{p_\infty}$ at the limit $p_\infty$ satisfies
\begin{align*}
\mathcal{E}_{p_\infty}(\beta')
&=
\lim_{t \to\infty}
\mathcal{E}_{\psi_*(R_{\alpha,x_0}(t))}(\beta') \\
&=
\lim_{t \to\infty}
e^{-d_T(x_0,\psi_*(R_{\alpha,x_0}(t)))}\cdot 
\ext_{\psi_*(R_{\alpha,x_0}(t))}(\beta')^{1/2}=0
\end{align*}
for all $\beta'\in \mathcal{S}$ with $i(\psi(\alpha),\beta')=0$.
Since the set $\{t\beta'\in \mathbb{R}_+\otimes \mathcal{S}\mid i(\psi(\alpha),\beta')=0\}$
is dense in $\mathcal{N}(\psi(\alpha))\cap \mf$,
by Lemma \ref{lem:charcterization_ue},
the limit $p_\infty$ is equal to the projective class of $\psi(\alpha)$.
\qed
\end{proof}

\subsubsection{Proof of Corollary \ref{coro:Ivanov_characterization}}
\label{subsub:Proof_corollary_Ivanov}
Let $\omega$ be an isometry of $\Teich{g}{m}$.
Then,
$\omega$ extends homeomorphically to $\cl{\Teich{g}{m}}$
(cf. \cite{LiuSu}).
We denote by the same symbol $\omega$ the extension.
By Theorem \ref{thm:almost_isometry}
and Theorems by Ivanov, Korkmaz and Luo
in \cite{Ivanov0}, \cite{Korkmaz} and \cite{Luo},
there is a diffeomorphism $h$ on $X$ which induces the action of the complex of curves above.
%
%
By Lemma \ref{lem:isometric_action_extended_mapping_class_group}
$h$ acts on $\Teich{g}{m}$ isometrically
and the action extends on $\cl{\Teich{g}{m}}$.
We denote by $h_*$ the action of $h$ to $\cl{\Teich{g}{m}}$.
Let $\overline{\omega}=\omega\circ h_*^{-1}$.
By Lemma \ref{lem:action_bdy},
$\overline{\omega}$ acts on $\Teich{g}{m}$ isometrically
and coincides with the identity on
$\pmf\subset \partialGM{\Teich{g}{m}}$.

The following argument is impressed with 
the proof of Theorem A in \cite{Ivanov}.
However,
our situation is different from that in Ivanov's proof
as we mentioned in \S\ref{subsubsec:isometry-teich}.
For completeness,
we proceed to prove the theorem.

\begin{claim}
\label{claim:fixed_pt}
$\overline{\omega}$ has a fixed point in $\Teich{g}{m}$.
\end{claim}

\begin{proof}
Take $\alpha,\beta\in \mathcal{S}$ which fill up $X$.
Consider a holomorphic quadratic differential $q$
whose horizontal and vertical foliations are $\alpha$ and $\beta$
respectively (cf. \cite{HM}).
Consider the Teichm\"uller disk $\varphi:\mathbb{D}\to \Teich{g}{m}$
corresponding to the quadratic differential $q$.
It is well-known that the Teichm\"uller disk $\varphi$
is invariant under the action of a pseudo-Anosov
mapping $\tau_\alpha\circ \tau_\beta^{-1}$
where $\tau_\alpha$ and $\tau_\beta$ are
Dehn-twists along $\alpha$ and $\beta$,
respectively (cf. \cite{Thurston1}).
Let $\mu_1$ and $\mu_2$ be the stable and unstable foliations
of the pseudo-Anosov mapping.
For simplifying of the notation,
we set $\{\lambda_i\}_{i=1}^4=\{\alpha,\beta,\mu_1,\mu_2\}$,
where the equality holds as unordered sets.
Let $\theta_i\in \partial\mathbb{D}$
be the corresponding point to $\lambda_i$ via $\varphi$.
This means that the radial ray of direction $\theta_i$
terminates at the projective class of $\lambda_i\in \partialGM{\Teich{g}{m}}$
(cf. \cite{Mi3}. See also Theorem 5.1 of \cite{GM} and Lemma \ref{lem:charcterization_ue}).
We may assume that $\theta_i$ lies on $\partial \mathbb{D}$
counterclockwise.
For $i=1,2$,
let $g_i$
be the hyperbolic geodesic connecting $\theta_{i}$ and $\theta_{i+2}$
in $\mathbb{D}$.
Then,
$g_1$ and $g_2$ intersect transversely in $\mathbb{D}$,
%
and $\varphi(g_1)\cap \varphi(g_2)$ consists of one point,
say $x_1\in \Teich{g}{m}$  since $\varphi$ is injective.

Since each end of $g_i$ are asymptotically tangent to the radial ray
at $\partial \mathbb{D}$,
$\varphi(g_i)$ is Teichm\"uller geodesic
which terminates at the projective classes of $\lambda_i$
and $\lambda_{i+2}$ in the Gardiner-Masur compactification
(cf. \cite{LiuSu} and Proposition 4.9 in \cite{Rieffel}).
Notice from Theorem 1.1 in \cite{Mi3}
that the limits of two different Teichm\"uller rays emanating from $x_1$
are different in the Gardiner-Masur compactification.
Hence,
the horizontal and vertical foliations of
corresponding quadratic differential $q_i$ should be
$\lambda_{i}$ and $\lambda_{i+2}$ for $i=1,2$.

Since $\overline{\omega}$ is the identity on $\pmf$,
$\overline{\omega}(\varphi(g_i))$ is also a Teichm\"uller geodesic
terminating at the projective classes of $\lambda_i$ and $\lambda_{i+2}$.
By applying Theorem 1.1 in \cite{Mi3} as above,
we deduce that
$\overline{\omega}(\varphi(g_i))$ is the Teichm\"uller geodesic
of the holomorphic quadratic differential whose horizontal
and vertical foliations are $\lambda_i$ and $\lambda_{i+2}$.
Thus,
by Theorem 5.1 in \cite{GM},
$\overline{\omega}(\varphi(g_i))=\varphi(g_i)$ for $i=1,2$
and hence $\overline{\omega}$ fixes the intersecting point $x_1$.
\qed
\end{proof}

\begin{claim}
\label{claim:identity}
$\overline{\omega}$ is the identity
on $\Teich{g}{m}$.
\end{claim}

\begin{proof}
%
As in the previous section,
for $\alpha\in \mathcal{S}$,
we denote by $R_{\alpha,x_1}\colon [0,\infty)\to \Teich{g}{m}$
the Teichm\"uller geodesic ray
which  emanates from $x_1$ and is defined by the Jenkins-Strebel differential
on $x_1$ whose vertical foliation is $\alpha$.
Hence,
from Theorem 1.1 in \cite{Mi3} again,
we have that $R_{\alpha,x_1}$ is
the only geodesic ray which emanates from $x_1$
and terminates at $[\alpha]\in \pmf\subset \partialGM{\Teich{g}{m}}$
since $\lim_{t\to \infty}R_{\alpha,x_1}(t)=[\alpha]$ by Theorem 5.1 of \cite{GM}.
Since $\overline{\omega}([\alpha])=[\alpha]$,
we deduce that $\overline{\omega}\circ R_{\alpha,x_1}=R_{\alpha,x_1}$ on $[0,\infty)$.
Since Teichm\"uller rays $\{R_{\alpha,x_1}\}_{\alpha\in \mathcal{S}}$
are dense in $\Teich{g}{m}$,
we conclude that $\overline{\omega}$ is the identity on $\Teich{g}{m}$.
\qed
\end{proof}

For closing the proof of Corollary \ref{coro:Ivanov_characterization},
we check that the extended mapping class group $\mcg^*(X)$ is isomorphic
to the isometry group ${\rm Isom}(\Teich{g}{m},d_T)$ of $(\Teich{g}{m},d_T)$.
From Lemma \ref{lem:isometric_action_extended_mapping_class_group},
there is a natural homomorphism
\begin{equation}
\label{eq:homomorphism}
\mcg^*(X)\ni h\mapsto h_*\in {\rm Isom}(\Teich{g}{m},d_T).
\end{equation}
From Claim \ref{claim:identity},
the homomorphism \eqref{eq:homomorphism} is surjective.
Let $h\in \mcg^*(X)$
and assume that $h_*=id$ on $\Teich{g}{m}$.
Then,
from Lemma \ref{lem:action_bdy},
the extension of $h_*$ to $\partialGM{\Teich{g}{m}}$
fixes $\mathcal{S}$ pointwise.
From Theorem \ref{thm:almost_isometry},
$h_*$ induces the identity automorphism of the complex of curves.
Hence,
by Ivanov-Korkmaz-Luo's theorem,
$h$ should be the identity
from the topological assumption of $X$.
\qed

\subsection{Comments on the exceptional cases}
Suppose first that
$(g,m)=(1,2)$.
It is known that the canonical homomorphism
from the extended mapping class group on $X_{1,2}$ to the isometry group 
is neither injective nor surjective.
Indeed,
by Proposition 1.3 in \cite{EK},
$\Teich{1}{2}$ admits a biholomorphic mapping
to the Teichm\"uller space $\Teich{0}{5}$ of
a sphere $X_{0,5}$ with five punctures
which is induced by the quotient mapping $X_{1,2}\to X_{0,5}$ of
the action of the hyperelliptic involution (double branched points are
considered as punctures).
Hence,
from Corollary \ref{coro:Ivanov_characterization},
the isometry group of $\Teich{1}{2}$
is isometric to the extended mapping class group
$\mcg^*(X_{0,5})$ of $X_{0,5}$
since the Teichm\"uller distance coincides with the Kobayashi distance.
Therefore,
the canonical 
homomorphism from the extended mapping class group $\mcg^*(X_{1,2})$
to the isometry group of $\Teich{1}{2}$ is not surjective
(cf. Corollary 3 in \S4.3 of \cite{EK}).
By a theorem due (independently) to Birman and Viro,
the hyperelliptic involution of $X_{1,2}$ fixes every non-trivial and non-peripheral
simple closed curves on $X_{1,2}$ (cf. \cite{Luo}).
Hence,
the hyperelliptic involution acts trivially on $\Teich{1}{2}$
and the canonical homomorphism is not injective (cf. \cite{EK}).

When
$(g,m)=(2,0)$,
any automorphism of the complex of curves induces a homeomorphism on $X_{2,0}$.
However,
the hyperelliptic involution fixes every non-trivial simple closed curves on $X_{2,0}$
and hence the action of the extended mapping class group is not faithful
(cf. e.g. \S\ref{subsub:Proof_corollary_Ivanov}) .
In fact,
it is known that the hyperelliptic involution generates the kernel of the canonical homomorphism
(e.g. \cite{Luo})

\subsection{Comments on the characterization of biholomorphisms}
The problem of characterizing isometries and biholomorphisms
makes sense
for Teichm\"uller spaces of arbitrary Riemann surfaces.
In the case where the Teichm\"uller space is of infinite dimension,
Earle and Gardiner \cite{EG}
obtained the characterization
for Riemann surfaces of topologically finite type.
In \cite{Lakic},
N. Lakic obtained the characterization for Riemann surfaces
of finite genus.
Finally,
in \cite{Markovic},
Markovic settled the characterization for biholomorphisms
of Teichm\"uller space of arbitrary Riemann surfaces.

\section{Appendix : A proper geodesic metric space without extendable Gromov product}
\label{sec:non-extendability-GromovProduct}
This section is devoted to giving a geodesic metric space on which the Gromov product
does not extend to the horofunction boundary.
The following example is given by Cormac Walsh (cf. \cite{Walsh}).
Notice that the Gardiner-Masur compactfication
coincides with the horofunction compactification
with respect to the Teichm\"uller distance (cf. \cite{LiuSu}).

Let $C_{n}$ be the frame $\partial ([-n,n]\times [0,n])$
with the standard Euclidean metric.
We construct a space $X$
by gluing each frame $C_{n}$ to $\mathbb{R}$ along
the bottom edge $[-n,n]\times \{0\}$ of $C_{n}$ and the
interval $[-n,n]$ of $\mathbb{R}$
isometrically.
The space $X$ is a proper geodesic space
(cf. Figure \ref{fig:metricspaceX}).
\begin{figure}
\begin{center}
\includegraphics[height=4cm]{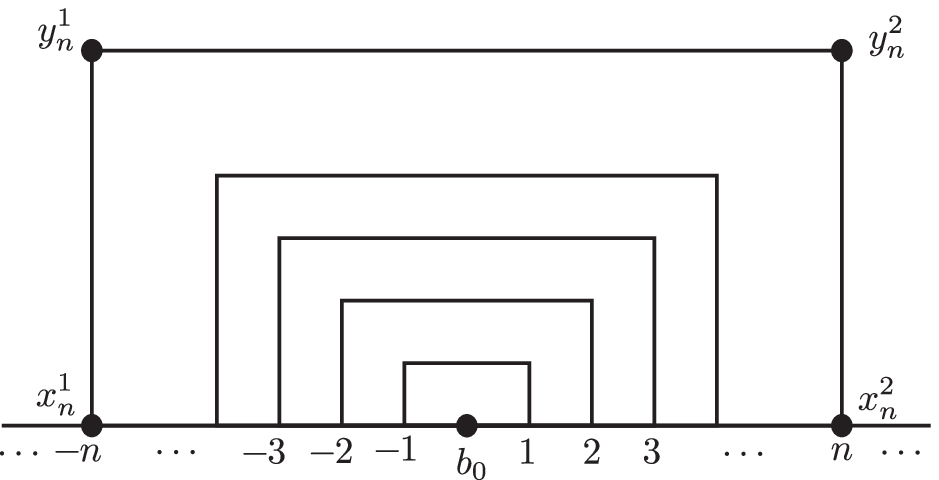}
\caption{The metric space $X$.}
\label{fig:metricspaceX}
\end{center}
\end{figure}
Let $b_{0}$,
$x^{1}_{n},y^{1}_{n},x^{2}_{n}$ and $y^{2}_{n}$ be points in $X$
corresponding to $0\in \mathbb{R}$,
$(-n,0)$,
$(-n,n)$,
$(n,0)$ and $(n,n)$ in $C_{n}$ respectively.
We consider $b_{0}$ as the basepoint of $X$.
Then,
one can see that for $i=1,2$,
$\{x^{i}_{n}\}_{n}$ and $\{y^{i}_{n}\}_{n}$
converges to the same Busemann point in the horofunction boundary of $X$
though $\{y^{i}_{n}\}_{n}$ is not an almost geodesic (cf. \cite{Rieffel}).
On the other hand,
we see
$$
\lim_{n\to \infty}
\gromov{y^{1}_{n}}{y^{2}_{n}}{b_{0}}
=\lim_{n\to \infty}\frac{1}{2}(2n+2n-2n)=\infty
$$
while $\gromov{x^{1}_{n}}{x^{2}_{n}}{b_{0}}=(n+n-2n)/2=0$ for all $n$.

\medskip
\noindent
\paragraph{{\bf Acknowledgements}}
The author thanks Professor Ken'ichi Ohshika and
Professor
Athanase Papadopoulos
for stimulating and useful conversations
and continuous encouragements.
The author would like to express his heartfelt gratitude 
to Professor Francis Bonahon for his valuable suggestions and discussions
and for his kind hospitality in the author's visit at USC.
The author thanks Professor Cormac Walsh for informing his example and
for kindly permitting to put it in this paper.
Finally,
he is also grateful to the referee for his/her careful reading
and for a number of helpful suggestions.
\end{document}